\documentclass[a4paper,12pt,reqno]{amsart}  

\usepackage{amsmath}
\usepackage{tikz-cd}
\usepackage{amssymb}
\usepackage{amsthm}
\usepackage{amsfonts}
\usepackage{mathtools}
\usepackage{comment}
\usepackage{mathrsfs}
\usepackage{pifont}
\usepackage{cite}
\usepackage{enumerate}
\usepackage{here}
\usepackage{autobreak}
\usepackage{wrapfig}
\usepackage{graphicx}
\usepackage{lscape}
\usepackage[colorlinks]{hyperref}  
\usepackage{xcolor}
\hypersetup{
	bookmarksnumbered=true,
    colorlinks=true,
    citecolor=blue,
    linkcolor=purple,
    urlcolor=orange,
}
\usepackage{mleftright} 
\usepackage{extpfeil} 
\usepackage{tikz}
\usetikzlibrary{arrows} 
\usetikzlibrary{cd, decorations.pathmorphing} 

\usepackage{pgfplots}
\pgfplotsset{compat=1.18}

\usepackage[capitalize, nameinlink]{cleveref}

\everymath{\displaystyle}
\numberwithin{equation}{section}

\theoremstyle{plain}

\newtheorem{thm}{Theorem}[section]
\crefname{thm}{Theorem}{Theorems}

\newtheorem{lem}[thm]{Lemma}
\crefname{lem}{Lemma}{Lemmas}
\newtheorem{prop}[thm]{Proposition}
\crefname{prop}{Proposition}{Propositions}
\newtheorem{cor}[thm]{Corollary}
\crefname{cor}{Corollary}{Corollaries}

\newtheorem*{claim*}{Claim}
\newtheorem*{thm*}{Theorem}

\newtheorem{introthm}{Theorem}[section]

\crefname{introthm}{Theorem}{Theorems}
\crefname{introprop}{Proposition}{Propositions}
\crefname{introcor}{Corollary}{Corollaries}

\theoremstyle{definition}

\newtheorem{dfn}[thm]{Definition}
\crefname{dfn}{Definition}{Definitions}

\newtheorem{setup}[thm]{Setup}
\crefname{setup}{Setup}{Setups}
\newtheorem{eg}[thm]{Example}
\crefname{eg}{Example}{Examples}
\newtheorem*{Ack}{Acknowledgement}
\newtheorem*{NoCon}{Notation and Conventions}
\newtheorem*{Out}{Outline of the paper}

\newtheorem*{Rels}{Related Works}
\newtheorem*{AI}{AI Usage}

\theoremstyle{remark}
\newtheorem{rem}[thm]{Remark}
\crefname{rem}{Remark}{Remarks}

\usepackage{enumitem} 

\DeclareMathOperator{\Ext}{Ext}
\DeclareMathOperator{\ext}{ext}
\DeclareMathOperator{\Pic}{Pic}
\DeclareMathOperator{\Aut}{Aut}
\DeclareMathOperator{\Auteq}{Auteq}

\DeclareMathOperator{\Stab}{Stab}

\DeclareMathOperator{\ord}{ord}

\DeclareMathOperator{\Ker}{Ker}
\DeclareMathOperator{\ch}{ch}
\DeclareMathOperator{\td}{td}

\DeclareMathOperator{\id}{id}

\DeclareMathOperator{\GL}{GL}
\DeclareMathOperator{\SL}{SL}

\DeclareMathOperator{\rad}{rad}
\DeclareMathOperator{\End}{End}
\DeclareMathOperator{\Supp}{Supp}
\DeclareMathOperator{\Tr}{Tr}
\DeclareMathOperator{\sTr}{sTr}

\DeclareMathOperator{\MH}{MH}

\DeclareMathOperator{\Alb}{Alb}
\DeclareMathOperator{\Gal}{Gal}

\newcommand\Hom{\mathop{\mathrm{Hom}}\nolimits}

\newcommand{\PP}[1]{\mathbb{P}^{#1}}
\newcommand{\thcl}[1]{\langle{#1}\rangle_{\mathrm{th}}}
\newcommand{\abs}[1]{\lvert{#1}\rvert}
\newcommand{\norm}[1]{\lVert{#1}\rVert}
\newcommand{\lderived}{\mathbf{L}}
\newcommand{\rderived}{\mathbf{R}}
\newcommand{\GLt}{\widetilde{\GL}^{+}(2,\RB)}

\newcommand\dual{\raise0.9ex\hbox{$\scriptscriptstyle\vee$}}

\newcommand{\htop}{h_{\mathrm{top}}}
\newcommand{\hcat}{h_{\mathrm{cat}}}
\newcommand{\numg}[1]{{#1}^{N}}

\newcommand{\vtime}{N}
\newcommand{\tensfunc}{\mathsf{M}}

\newcommand{\CB}{\mathbb{C}}

\newcommand{\RB}{\mathbb{R}}

\newcommand{\ZZ}{\mathbb{Z}}

\newcommand{\AC}{\mathcal{A}}

\newcommand{\GC}{\mathcal{G}}

\newcommand{\KC}{\mathcal{K}}

\newcommand{\PC}{\mathcal{P}}

\newcommand{\OO}{\mathscr{O}}

\author{Tomoki Yoshida}
\address[T.Yoshida]{Department~of~Mathematics, School~of~Science~and~Engineering, Waseda~University, Ohkubo~3-4-1, Shinjuku, Tokyo~169-8555, Japan}
\email{\href{mailto:tomoki_y@asagi.waseda.jp}{tomoki\_y@asagi.waseda.jp}}

\title[Orbits in stability manifolds and categorical entropy]{Orbits in Stability Manifolds and Categorical Entropy: Applications to Varieties with Finite Albanese Morphisms}
\date{September 22, 2026}

\usepackage[all]{xy} 
\usepackage{etoolbox} 

\AtBeginEnvironment{lem}{\crefalias{thm}{lem}}
\AtBeginEnvironment{prop}{\crefalias{thm}{prop}}
\AtBeginEnvironment{cor}{\crefalias{thm}{cor}}
\AtBeginEnvironment{dfn}{\crefalias{thm}{dfn}}
\AtBeginEnvironment{eg}{\crefalias{thm}{eg}}
\AtBeginEnvironment{rem}{\crefalias{thm}{rem}}
\AtBeginEnvironment{introprop}{\crefalias{introthm}{introprop}}
\AtBeginEnvironment{introcor}{\crefalias{introthm}{introcor}}

\begin{document}

\begin{abstract}
This paper proves the Gromov--Yomdin property for smooth projective varieties with finite Albanese morphisms.
Our approach is to study the orbits of stability conditions under the action of autoequivalences on Bridgeland stability manifolds.
For a finite generating set of the derived category, we consider the locus on which all its objects are semistable.
Using this locus, we formulate an \emph{orbit escape principle}, which provides a sufficient condition for an autoequivalence to satisfy the Gromov--Yomdin equality.
Together with the stability of simple semihomogeneous bundles with respect to arbitrary numerical stability conditions, this principle readily yields the result for abelian varieties.
We then reduce the case of varieties with finite Albanese morphisms to the abelian case by passing to suitable finite \'etale covers.

Finally, we construct autoequivalences with positive categorical entropy on every abelian variety and show that a smooth projective variety with a finite Albanese morphism admits such an autoequivalence if and only if it is not of general type.
\end{abstract}

\maketitle

\setcounter{tocdepth}{2} 
\tableofcontents

\setcounter{section}{-1} 
\section{Introduction}\label{section: introduction}

Let $X$ be a smooth projective variety over $\CB$ and $\Auteq(X)$ be the group of exact autoequivalences of the bounded derived category $D^b(X)$ of coherent sheaves on $X$.
This group contains the subgroup of \emph{standard autoequivalences}
\[
A(X)\coloneqq \Pic(X)\rtimes\Aut(X)\times\ZZ[1],
\]
and may have a much richer structure than $\Aut(X)$.
In this paper, we study $\Auteq(X)$ from the point of view of dynamics.

The \emph{topological entropy} $\htop(f)$ of a dynamical system $(X,f)$ measures the complexity of the iterates of $f$.
It plays an important role in the study of automorphisms of infinite order.
In \cite{dimitrov_haiden_katzarkov_kontsevich_2014_dynamical_systems_and_categories}, Dimitrov, Haiden, Katzarkov, and Kontsevich introduced \emph{categorical entropy} as an analogue of topological entropy for autoequivalences of triangulated categories.
We write $\hcat(\Phi)$ for the categorical entropy of $\Phi\in\Auteq(X)$ (see \cref{definition: categorical complexity and entropy}).
Subsequent studies have shown that categorical entropy shares several properties with topological entropy, while also exhibiting important differences.

One of the fundamental results concerning topological entropy is the \emph{Gromov--Yomdin equality} \cite{gromov_1987_entropy_homology_and_semialgebraic_geometry,gromov_2003_on_the_entropy_of_holomorphic_maps,yomdin_1987_volume_growth_and_entropy}.
For every $f\in\Aut(X)$, it asserts that
\[
\htop(f) = \log\rho(f^*\rvert_{\oplus_{p=0}^{\dim X}H^{p,p}(X)}),
\]
where $\rho(-)$ denotes the spectral radius.
It is natural to ask whether an analogous formula holds for categorical entropy.

For $\Phi\in\Auteq(X)$, we denote by $\numg{\Phi}$ the induced action on the numerical Grothendieck group $N(X)_{\RB}$ (see \cref{subsection: numerical actions}).
We say that $X$ has the \emph{Gromov--Yomdin (GY) property} if the following condition holds:
\begin{enumerate}[label={\textbf{(GY)}}, ref={(GY)}]
    \item \label{property: Gromov-Yomdin property}
    \begin{center}
        for all $\Phi\in\Auteq(X)$, $\hcat(\Phi) = \log\rho(\numg{\Phi})$.
    \end{center}
\end{enumerate}
The inequality $\log\rho(\numg{\Phi})\leq\hcat(\Phi)$ was proved by Kikuta, Shiraishi, and Takahashi \cite{kikuta_shiraishi_takahashi_2020_a_note_on_entropy_of_autoequivalences_lower_bound_and_the_case_of_orbifold_projective_lines} (see \cref{theorem: numerical lower bound for categorical entropy}).
Thus, the reverse inequality is the essential part of the Gromov--Yomdin property.

As we review below, some classes of varieties satisfy the Gromov--Yomdin property, whereas others do not.
These results suggest that the difference between categorical entropy and the logarithm of the numerical spectral radius reflects information not captured by the numerical action alone.

To study this difference, we consider another measure of complexity provided by \emph{Bridgeland stability conditions} \cite{bridgeland_2007_stability_conditions_on_triangulated_categories} (see \cref{definition: stability condition}).
For a stability condition $\sigma$, the \emph{mass} $m_\sigma(E)$ of an object $E$ is the sum of the absolute values of the central charges of its Harder--Narasimhan factors (see \cref{subsection: stability conditions}).
This quantity depends on the semistable factors of $E$, rather than only on its numerical class.

We write $\Stab_N(X)$ for the space of numerical stability conditions satisfying the support property with respect to the full numerical Grothendieck group $N(X)$.
Dimitrov--Haiden--Katzarkov--Kontsevich \cite{dimitrov_haiden_katzarkov_kontsevich_2014_dynamical_systems_and_categories}
also introduced the notion of \emph{mass growth} $h_\sigma(\Phi)$, which measures the exponential growth of the masses of objects under iteration of $\Phi$ (see \cref{definition: mass growth}).
Ikeda \cite{ikeda_2021_mass_growth_of_objects_and_categorical_entropy} studied its fundamental properties and proved
\[
    \log\rho(\numg{\Phi})\leq h_\sigma(\Phi)\leq\hcat(\Phi)
\]
for every $\Phi\in\Auteq(X)$ and $\sigma\in\Stab_N(X)$ (see \cref{theorem: mass estimates and generator formula}).

The purpose of this paper is to compare these invariants and to relate their equality to the action of autoequivalences on stability conditions by considering the orbits of the action $\Auteq(X)\curvearrowright\Stab_N(X)$.

\subsection{Results}

First, we show that a mass--Hom bound implies the equality between mass growth and categorical entropy.
This bound was introduced by Halpern-Leistner and Robotis \cite{preprint_daniel_antoniosalexandros_2025_the_space_of_augmented_stability_conditions} (see \cref{definition: mass Hom bound}).
Let $\MH_X\subset\Stab_N(X)$ denote the set of numerical stability conditions satisfying this property.
Recent studies by Li--Liu--Liu--Macr\`i--Perry--Stellari--Zhao \cite{preprint_chunyi_zhiyu_ziqi_emanuele_alexander_paolo_xiaolei_2026_stability_conditions_and_moduli_spaces_on_projective_families} and Cheng \cite{preprint_yiran_2026_a_remark_on_the_full_support_property} proved that $\MH_X\neq\varnothing$ for every smooth projective variety $X$ (see \cref{theorem: existence of numerical mass Hom stability conditions}).

We write $\delta_t$ and $h_t$ for categorical complexity and entropy, and $m_{\sigma,t}$ and $h_{\sigma,t}$ for weighted mass and mass growth (see \cref{definition: categorical complexity and entropy,subsection: stability conditions,definition: mass growth}).
Extending the mass--Hom bound to weighted total Hom spaces (see \cref{proposition: weighted total Hom mass bound}), we obtain, for every stability condition $\sigma$ satisfying this bound, every classical generator $G$, and every $t\in\RB$, constants $c,C>0$ such that
\[
    c\,m_{\sigma,t}(E)\leq\delta_t(G,E)\leq C\,m_{\sigma,t}(E)
\]
for every $E\in D^b(X)$.
This yields the following result.

\begin{introthm}[See \cref{corollary: equality of entropy and mass growth}]
\label{theorem in intro: comparison of entropy and mass growth}
    Let $X$ be a smooth projective variety, and let $\sigma$ be a stability condition on $D^b(X)$ satisfying a mass--Hom bound.
    For every Fourier--Mukai endofunctor $\Phi$ of $D^b(X)$, every classical generator $G$, and every $t\in\RB$, the following limit exists and satisfies
    \[
        h_t(\Phi)=h_{\sigma,t}(\Phi)
        =\lim_{n\to\infty}\frac{1}{n}\log m_{\sigma,t}(\Phi^nG).
    \]
\end{introthm}

Note that mass growth is defined using a $\limsup$ (see \cref{definition: mass growth}).
Under the mass--Hom hypothesis, the two-sided comparison above allows us to replace the $\limsup$ in the generator formula by a limit.

Next, we relate the entropy of an autoequivalence to its action on the stability manifold.
Let $\GC=\{G_1,\ldots,G_r\}$ be a finite set of nonzero objects of $D^b(X)$ with $\thcl{\GC}=D^b(X)$, and set
\[
    \Sigma_{\GC}=\{\tau\in\Stab_N(X)\mid G_i\text{ is }\tau\text{-semistable for every }i\}
\]
(see \cref{definition: semistability locus and return times}).
Put $G=\bigoplus_{i=1}^rG_i$.
For every $\Phi\in\Auteq(X)$ and $\sigma\in\Stab_N(X)$, the covariance of mass gives $m_\sigma(\Phi^nG)=m_{\Phi^{-n}\sigma}(G)$.
Thus, the mass growth of $\Phi$ can be studied along the negative orbit of $\sigma$.
Our second observation is the following, which we call the \emph{orbit escape principle}.

\begin{introthm}[See \cref{theorem: orbit escape principle}]
\label{theorem in intro: orbit escape principle}
    Let $\Phi\in\Auteq(X)$ and $\sigma\in\MH_X$.
    If
    \[
        N_{\GC,\sigma}(\Phi)\coloneqq\{n\in\ZZ_{\geq0}\mid\Phi^{-n}\sigma\in\Sigma_{\GC}\}
    \]
    is infinite, then
    \[
        \hcat(\Phi)=h_\sigma(\Phi)=\log\rho(\numg{\Phi}).
    \]
    Equivalently, if $\hcat(\Phi)>\log\rho(\numg{\Phi})$, there is an integer $N$ such that $\Phi^{-n}\sigma\notin\Sigma_{\GC}$ for every $n\geq N$.
\end{introthm}

As an application of the orbit escape principle, we prove the GY property for abelian varieties of arbitrary dimension.
Fu--Li--Zhao \cite{fu_li_zhao_2022_stability_manifolds_of_varieties_with_finite_albanese_morphisms} proved that every simple semihomogeneous vector bundle on an abelian variety $A$ is stable with respect to every numerical stability condition (see \cref{theorem: universal stability for finite Albanese morphisms}).
A finite generating set $\GC$ of line bundles therefore satisfies $\Sigma_{\GC}=\Stab_N(A)$.
Together with $\MH_A\neq\varnothing$ and the general inequalities for mass growth, this yields the following result.

\begin{introthm}[See \cref{theorem: Gromov Yomdin equality for abelian varieties}]
\label{theorem in intro: Gromov Yomdin equality for abelian varieties}
    Let $A$ be an abelian variety.
    For every $\Phi\in\Auteq(A)$ and every $\sigma\in\Stab_N(A)$,
    \[
        \hcat(\Phi)=h_\sigma(\Phi)=\log\rho(\numg{\Phi}).
    \]
\end{introthm}

We next extend this equality to varieties with finite Albanese morphisms.

\begin{introthm}[See \cref{theorem: Gromov Yomdin equality for finite Albanese morphisms}]
\label{theorem in intro: Gromov Yomdin equality for finite Albanese morphisms}
    Let $X$ be a smooth projective variety whose Albanese morphism is finite.
    For every $\Phi\in\Auteq(X)$ and every $\sigma\in\Stab_N(X)$,
    \[
        \hcat(\Phi)=h_\sigma(\Phi)=\log\rho(\numg{\Phi}).
    \]
\end{introthm}

To prove this result, we use Kawamata's finite \'etale Galois cover $p:B\times Y\to X$, where $B$ is an abelian variety of dimension $\dim X-\kappa(X)$ and $Y$ is of general type with finite Albanese morphism \cite{kawamata_1981_characterization_of_abelian_varieties,jiang_2011_an_effective_version_of_a_theorem_of_kawamata_on_the_albanese_map}.
After replacing $\Phi$ by a positive power, we lift it to $B\times Y$ and obtain a normal form involving an autoequivalence $\Psi\in\Auteq(B)$ and standard autoequivalences (see \cref{proposition: lifting a power to the product cover,proposition: relative normal form}).
Comparing categorical entropy and numerical spectral radii then reduces the assertion to the abelian case established in \cref{theorem in intro: Gromov Yomdin equality for abelian varieties} (see \cref{theorem: reduction of entropy to abelian factor}).

Finally, we study which varieties admit autoequivalences of positive categorical entropy.
For smooth projective surfaces, Cantat \cite{cantat_1999_dynamique_des_automorphismes_des_surfaces_projectives_complexes} showed that the existence of an automorphism of positive topological entropy forces the surface to be birational to the projective plane, a K3 surface, an Enriques surface, or an abelian surface.
This motivates the analogous question for categorical entropy.
Constructions using spherical twists show that a surface may admit an autoequivalence of positive categorical entropy even if it has no automorphism of positive topological entropy \cite{mattei_2021_categorical_vs_topological_entropy_of_autoequivalences_of_surfaces}.
The author established this phenomenon for bielliptic surfaces, which have neither spherical objects nor automorphisms of positive topological entropy \cite{yoshida_2026_a_note_on_categorical_entropy_of_bielliptic_surfaces_and_enriques_surfaces}.

We construct positive-entropy autoequivalences on abelian varieties with prescribed equivariance under finite translation groups, and descend them to obtain the following characterization.

\begin{introthm}[See \cref{proposition: positive categorical entropy on abelian varieties,theorem: positive categorical entropy for finite Albanese morphisms}]
\label{theorem in intro: existence results of positive categorical entropy}
    Let $X$ be a smooth projective variety whose Albanese morphism is finite.
    Then, there exists an autoequivalence $\Phi\in\Auteq(X)$ with $\hcat(\Phi)>0$ if and only if $\kappa(X)<\dim X$.

    In particular, every abelian variety admits an autoequivalence with positive categorical entropy.
\end{introthm}

\begin{Rels}
The Gromov--Yomdin property is known for curves \cite{kikuta_2017_on_entropy_for_autoequivalences_of_the_derived_category_of_curves}, orbifold projective lines \cite{kikuta_shiraishi_takahashi_2020_a_note_on_entropy_of_autoequivalences_lower_bound_and_the_case_of_orbifold_projective_lines}, abelian surfaces and simple abelian varieties \cite{yoshioka_2020_categorical_entropy_for_fouriermukai_transforms_on_generic_abelian_surfaces}, and bielliptic surfaces \cite{yoshida_2026_a_note_on_categorical_entropy_of_bielliptic_surfaces_and_enriques_surfaces}.
Kikuta and Takahashi \cite{kikuta_takahashi_2019_on_the_categorical_entropy_and_the_topological_entropy} proved it for varieties with ample or anti-ample canonical bundle, and established the corresponding equality for pullbacks by surjective endomorphisms.
In contrast, counterexamples have been constructed on Calabi--Yau hypersurfaces \cite{fan_2018_entropy_of_an_autoequivalence_on_calabiyau_manifolds}, K3 surfaces \cite{ouchi_2020_on_entropy_of_spherical_twists}, surfaces containing a $(-2)$-curve \cite{mattei_2021_categorical_vs_topological_entropy_of_autoequivalences_of_surfaces}, and further Calabi--Yau varieties \cite{barbacovi_kim_2023_entropy_of_the_composition_of_two_spherical_twists} via a composition of spherical twists.
The author constructed counterexamples on Enriques surfaces \cite{yoshida_2026_a_note_on_categorical_entropy_of_bielliptic_surfaces_and_enriques_surfaces}, and on certain Hilbert schemes of points on surfaces and hyperk\"ahler manifolds \cite{preprint_tomoki_2026_categorical_entropies_of_hilbert_schemes_of_points_on_surfaces_and_hyperkahler_manifolds}.

Ikeda \cite[Theorem 3.14]{ikeda_2021_mass_growth_of_objects_and_categorical_entropy} proved that $h_t(\Phi)=h_{\sigma,t}(\Phi)$ on any connected component containing an algebraic stability condition.
Our comparison theorem obtains this equality under a mass--Hom bound.
Barbacovi and Kim \cite{barbacovi_kim_2023_on_gromovyomdin_type_theorems_and_a_categorical_interpretation_of_holomorphicity} formulated a categorical analogue of holomorphicity using compatible triples, and proved $h_\sigma(\Phi)=\log\rho(\numg{\Phi})$ under an additional spanning condition on the central charge.
Under the mass--Hom hypothesis, our orbit criterion requires only infinitely many returns to $\Sigma_{\GC}$, rather than compatibility; see \cref{remark: compatible triples and orbit escape}.

Fan, Filip, Haiden, Katzarkov, and Liu \cite{fan_filip_haiden_liu_2021_on_pseudoanosov_autoequivalences} studied mass-growth filtrations and pseudo-Anosov autoequivalences.
Kikuta \cite{kikuta_2023_curvature_of_the_space_of_stability_conditions} related DHKK pseudo-Anosov autoequivalences to hyperbolic isometries of stability spaces modulo the $\CB$-action.
Woolf \cite{woolf_2024_massgrowth_of_triangulated_autoequivalences} related mass growth to eventual displacement in the Bridgeland metric.
In contrast to metric displacement, our orbit escape principle concerns eventual departure from a fixed semistability locus, which need not be bounded.

Bapat, Deopurkar, and Licata \cite{preprint_asilata_anand_anthony_2011_a_thurston_compactification_of_the_space_of_stability_conditions} studied Thurston-type compactifications using projectivized mass functions.
Kikuta, Koseki, and Ouchi \cite{kikuta_koseki_ouchi_2024_thurston_compactifications_of_spaces_of_stability_conditions_on_curves} developed this approach for curves and obtained a Nielsen--Thurston-type classification for elliptic curves.
Further refinements of categorical dynamics include polynomial entropy and polynomial mass growth, introduced by Fan, Fu, and Ouchi \cite{fan_fu_ouchi_2021_categorical_polynomial_entropy}, and asymptotic shifting numbers, introduced by Fan and Filip \cite{fan_filip_2023_asymptotic_shifting_numbers_in_triangulated_categories}.
\end{Rels}

\begin{Out}
In \cref{section: preliminaries}, we recall categorical entropy, stability conditions, mass--Hom bounds, and the required facts about abelian varieties.
We prove the equality between mass growth and categorical entropy (\cref{theorem in intro: comparison of entropy and mass growth}) in \cref{section: equality of mass growth and categorical entropy}.
The orbit escape principle (\cref{theorem in intro: orbit escape principle}) is established in \cref{section: orbit escape principle}.
In \cref{section: abelian varieties}, we prove the Gromov--Yomdin type equality for abelian varieties (\cref{theorem in intro: Gromov Yomdin equality for abelian varieties}) and compare the numerical and cohomological spectral radii.

In \cref{section: finite Albanese morphisms}, we prove the GY property for varieties with finite Albanese morphisms (\cref{theorem in intro: Gromov Yomdin equality for finite Albanese morphisms}).
We lift a positive power of each autoequivalence to Kawamata's finite \'etale product cover of $X$ in \cref{subsection: lifting to the product cover} and, in \cref{subsection: relative normal form}, establish a normal form for a positive power of the lifted autoequivalence.
We then complete the proof in \cref{subsection: proof of Gromov Yomdin property for finite Albanese morphisms}.
Finally, in \cref{section: examples of positive categorical entropy}, we construct autoequivalences with positive categorical entropy on every abelian variety and prove \cref{theorem in intro: existence results of positive categorical entropy}.
\end{Out}

\begin{NoCon}
The conventions and notations used in this paper are listed below:
\begin{itemize}
\item For $E,F\in D^b(X)$, we write $\hom^i(E,F)=\ext^i(E,F)$ for the dimension of $\Hom(E,F[i])=\Ext^i(E,F)$, and put $\hom(E,F)=\hom^0(E,F)$.
\item All stability conditions considered in this paper are numerical stability conditions.
\item All functors are derived unless otherwise stated. However, the symbols $\lderived$ and $\rderived$ are occasionally used to emphasize that the functor is actually derived.
\item For a line bundle $L$ on $X$, we write $\tensfunc_L$ for the tensoring functor $(-)\otimes L$.
\end{itemize}
\end{NoCon}

\begin{AI}
The author used ChatGPT 5.6 Sol and 6 Astra throughout the preparation of this work, particularly for mathematical discussions concerning applications to varieties with finite Albanese morphisms (\cref{section: finite Albanese morphisms}), constructing an example of positive categorical entropy (\cref{section: examples of positive categorical entropy}), and for English proofreading. 
All mathematical arguments and proofs developed with AI assistance were carefully checked and reconstructed by the author.
The author take full responsibility for the mathematical content of this paper, including any errors.
\end{AI}

\begin{Ack}
The author thanks Professors Kohei Kikuta and Genki Ouchi for useful conversations and comments on the draft.
The author acknowledges support from OpenAI through the ChatGPT for Academic Researchers program.
\end{Ack}

\section{Preliminaries}\label{section: preliminaries}

In this section, we recall categorical entropy, Bridgeland stability conditions, and mass--Hom bounds. We also collect the facts about abelian varieties used in \cref{section: abelian varieties}.

\subsection{Categorical entropy}\label{subsection: categorical entropy}
We first recall the notion of a classical generator.

\begin{dfn}\label{definition: thick closure and classical generator}
For a collection $\GC$ of objects of $D^b(X)$, its \emph{thick closure} $\thcl{\GC}$ is the smallest full triangulated subcategory containing $\GC$ and closed under taking direct summands. A finite set $\GC=\{G_1,\ldots,G_r\}$ of nonzero objects is a \emph{finite generating set} if $\thcl{\GC}=D^b(X)$. An object $G$ is a \emph{classical generator}, or a \emph{split-generator}, if $\thcl{G}=D^b(X)$.
\end{dfn}

\begin{thm}[{\cite[Theorem 4]{orlov_2009_remarks_on_generators_and_dimensions_of_triangulated_categories}}]\label{theorem: existence of generator}
Let $L$ be a very ample line bundle on a smooth projective variety $X$ of dimension $d$. For every $k\in\ZZ$, the object
\[
\bigoplus_{j=0}^{d}L^{\otimes(j+k)}
\]
is a classical generator of $D^b(X)$.
\end{thm}

We use the definition of categorical entropy introduced in \cite{dimitrov_haiden_katzarkov_kontsevich_2014_dynamical_systems_and_categories}.

\begin{dfn}\label{definition: categorical complexity and entropy}
Let $A,E\in D^b(X)$ and $t\in\RB$. The \emph{complexity of $E$ with respect to $A$} is
\[
\delta_t(A,E)\coloneqq\inf\left\{\sum_{i=1}^{k}e^{n_it}\ \middle|\
\begin{xy}
(0,5)*{0}="0", (20,5)*{E_1}="1", (33,5)*{\cdots},
(48,5)*{E_{k-1}}="k-1", (73,5)*{E\oplus E'}="k",
(10,-7)*{A[n_1]}="n1", (61,-7)*{A[n_k]}="nk",
\ar "0";"1" \ar "1";"n1" \ar@{.>} "n1";"0"
\ar "k-1";"k" \ar "k";"nk" \ar@{.>} "nk";"k-1"
\end{xy}
\right\}.
\]

For a Fourier--Mukai endofunctor $\Phi$ and classical generators $G,G'$ of $D^b(X)$, its \emph{entropy} is defined to be 
\[
h_t(\Phi)\coloneqq\lim_{n\to\infty}\frac{1}{n}\log\delta_t(G,\Phi^nG').
\]
Moreover, we write $\hcat(\Phi)\coloneqq h_0(\Phi)$ and call it the \emph{categorical entropy} of $\Phi$.
\end{dfn}

For $A,E\in D^b(X)$, the spaces $\Hom(A,E[i])$ are finite-dimensional and vanish for all but finitely many $i\in\ZZ$. Categorical entropy can be computed from these Hom spaces as follows.

\begin{thm}[{\cite[Lemma 2.5 and Theorem 2.6]{dimitrov_haiden_katzarkov_kontsevich_2014_dynamical_systems_and_categories}}]\label{theorem: categorical entropy and total Hom comparison}
Let $X$ be a smooth projective variety and let $t\in\RB$.
\begin{enumerate}
\item For every classical generator $G$ of $D^b(X)$, there are constants $a(G,t),b(G,t)>0$ such that
\[
a(G,t)\delta_t(G,E)\leq\sum_{i\in\ZZ}\hom^i(G,E)e^{-it}\leq b(G,t)\delta_t(G,E)
\]
for every $E\in D^b(X)$.
\item For every Fourier--Mukai endofunctor $\Phi$ and any two classical generators $G,G'$, the limit defining $h_t(\Phi)$ exists, is independent of $G,G'$, and satisfies
\[
h_t(\Phi)=\lim_{n\to\infty}\frac{1}{n}\log\left(\sum_{i\in\ZZ}\hom^i(G,\Phi^nG')e^{-it}\right).
\]
Its value belongs to $[-\infty,\infty)$. If $\Phi$ is an autoequivalence, then $\hcat(\Phi)\geq0$.
\end{enumerate}
\end{thm}

\begin{prop}[{\cite[Section 2]{dimitrov_haiden_katzarkov_kontsevich_2014_dynamical_systems_and_categories}}]\label{proposition: basic properties of categorical entropy}
Let $\Phi$ be a Fourier--Mukai endofunctor of $D^b(X)$, $r>0$ be an integer, and $s\in\ZZ$. 
Then, the following hold:
\begin{enumerate}
    \item $h_t(\Phi^r)=rh_t(\Phi)$.
    \item $h_t([s]\circ\Phi)=h_t(\Phi)+st$ for every $t\in\RB$. 
    \item If $\Psi:D^b(X)\to D^b(Y)$ is an equivalence, then $h_t(\Psi\Phi\Psi^{-1})=h_t(\Phi)$.
    \item For commuting Fourier--Mukai autoequivalences $\Phi_1,\Phi_2$, $h_t(\Phi_1\Phi_2)\leq h_t(\Phi_1)+h_t(\Phi_2)$.
\end{enumerate}
\end{prop}

\begin{lem}[{\cite[Lemma 3.10]{kikuta_takahashi_2019_on_the_categorical_entropy_and_the_topological_entropy}}]\label{lemma: categorical entropy of tensor functors}
For every line bundle $L$ on $X$, one has $\hcat(\tensfunc_L)=0$.
\end{lem}

The following estimate will be used to compare the growth of fixed objects with that of a generator.

\begin{lem}\label{lemma: fixed objects and total Hom growth}
Let $G$ be a classical generator of $D^b(X)$, and $A,B\in D^b(X)$. There is a constant $C(A,B,G)>0$ such that, for every exact endofunctor $\Phi$ and every $n\geq0$,
\[
\sum_{i\in\ZZ}\hom^i(A,\Phi^nB)\leq C(A,B,G)\sum_{i\in\ZZ}\hom^i(G,\Phi^nG).
\]
\end{lem}
\begin{proof}
Choose finite constructions of $A$ and $B$ from shifts of $G$ by exact triangles and direct summands, and apply $\Phi^n$ to the construction of $B$. The long exact sequences of Hom spaces give the estimate. Shifts do not change the total Hom dimension, and the two fixed constructions are independent of $\Phi$ and $n$.
\end{proof}

\subsection{Numerical actions and Fourier--Mukai transforms}\label{subsection: numerical actions}
We recall the numerical and cohomological actions of Fourier--Mukai functors; see \cite[Chapter 5]{book_huybrechts_2006_fouriermukai_transforms_in_algebraic_geometry}.
The Euler form on $D^b(X)$ is $\chi_X(E,F)=\sum_{i\in\ZZ}(-1)^i\hom^i(E,F)$. The \emph{numerical Grothendieck group} is $N(X)\coloneqq K_0(X)/\rad\chi_X$, where
\[
\rad\chi_X=\{v\in K_0(X)\mid\chi_X(v,w)=\chi_X(w,v)=0\text{ for every }w\in K_0(X)\}.
\]
It is a free abelian group of finite rank. We write $N(X)_{\RB}=N(X)\otimes_{\ZZ}\RB$. Every autoequivalence $\Phi$ induces an automorphism $\numg{\Phi}$ of $N(X)$, for which we use the same notation after extending scalars.

\begin{thm}[{\cite[Theorem 2.13]{kikuta_shiraishi_takahashi_2020_a_note_on_entropy_of_autoequivalences_lower_bound_and_the_case_of_orbifold_projective_lines}}]\label{theorem: numerical lower bound for categorical entropy}
For every $\Phi\in\Auteq(X)$, one has
\[
\log\rho(\numg{\Phi})\leq\hcat(\Phi).
\]
\end{thm}

For $K\in D^b(X\times Y)$, let $\Phi_K(E)=p_{Y*}(p_X^*E\otimes K)$ be the Fourier--Mukai functor with kernel $K$. Its cohomological action is
\[
\Phi_K^H(\alpha)=p_{Y*}\bigl(p_X^*\alpha\cdot\ch(K)\sqrt{\td(X\times Y)}\bigr).
\]
This action preserves even and odd cohomology, but need not preserve each cohomological degree. For an autoequivalence $\Phi$, the space of algebraic Mukai classes is invariant under $\Phi^H$, and $N(X)_{\CB}$ is a quotient of this space. 
Therefore, Riemann--Roch gives
\begin{equation}\label{equation: numerical and cohomological spectral radii}
\rho(\numg{\Phi})\leq\rho(\Phi^H).
\end{equation}
Here the spectral radius on the right is taken on the full cohomology $H^*(X,\CB)$.

\subsection{Stability conditions and mass growth}\label{subsection: stability conditions}
We use stability conditions with respect to a finite-rank lattice $\Lambda$ and a fixed class map $v:K_0(X)\to\Lambda$.

\begin{dfn}[{\cite{bridgeland_2007_stability_conditions_on_triangulated_categories}}]\label{definition: stability condition}
A \emph{stability condition} $\sigma=(Z,\PC)$ on $D^b(X)$ consists of a homomorphism $Z:\Lambda\to\CB$ and full additive subcategories $\PC(\phi)\subset D^b(X)$ for $\phi\in\RB$, satisfying the following conditions.
\begin{enumerate}
\item If $0\ne E\in\PC(\phi)$, then $Z(v(E))\in\RB_{>0}e^{i\pi\phi}$.
\item $\PC(\phi+1)=\PC(\phi)[1]$.
\item If $\phi_1>\phi_2$, then $\Hom(\PC(\phi_1),\PC(\phi_2))=0$.
\item Every nonzero object admits a finite Harder--Narasimhan filtration by exact triangles, with factors $A_j\in\PC(\phi_j)$ and $\phi_1>\cdots>\phi_\ell$.
\end{enumerate}
The slicing $\PC$ is required to be locally finite. Nonzero objects of $\PC(\phi)$ are called \emph{semistable of phase $\phi$}, and its simple objects are called \emph{stable}.
\end{dfn}

Throughout this paper, we also impose the \emph{support property}: for a norm on $\Lambda_{\RB}$, there is $C(\sigma)>0$ such that $\norm{v(A)}\leq C(\sigma)\abs{Z(v(A))}$ for every semistable object $A$.
For a nonzero object $E$ with Harder--Narasimhan factors, its \emph{weighted mass} is defined to be
\[
m_{\sigma,t}(E)\coloneqq\sum_{j=1}^{\ell}\abs{Z(v(A_j))}e^{t\phi_j}.
\]
We put $m_{\sigma,t}(0)=0$ and write $m_\sigma(E)=m_{\sigma,0}(E)$. The largest and smallest phases are denoted by $\phi^+_\sigma(E)$ and $\phi^-_\sigma(E)$.

We write $\Stab_N(X)$ for stability conditions with class map $K_0(X)\to N(X)$ and the support property for this full lattice. For each $\sigma\in\Stab_N(X)$, applying the support property to the Harder--Narasimhan factors gives a constant $C(\sigma)>0$ such that
\begin{equation}\label{equation: numerical norm bounded by mass}
\norm{[E]}\leq C(\sigma)m_\sigma(E)
\end{equation}
for every $E\in D^b(X)$.

The action of $\Phi\in\Auteq(X)$ on $\sigma=(Z,\PC)\in\Stab_N(X)$ is given by $\Phi\sigma=(Z\circ(\numg{\Phi})^{-1},\Phi\PC)$, where $(\Phi\PC)(\phi)=\Phi(\PC(\phi))$. Transport of Harder--Narasimhan filtrations gives
\begin{equation}\label{equation: covariance of mass}
m_{\Phi\sigma,t}(\Phi E)=m_{\sigma,t}(E).
\end{equation}
The functions $\phi^\pm_\sigma(E)$ and $m_\sigma(E)$ are continuous in $\sigma$. For a fixed nonzero object $E$, its stable locus is open and its semistable locus is closed, see \cite{bridgeland_2007_stability_conditions_on_triangulated_categories} for more detail.

\begin{dfn}[{\cite{dimitrov_haiden_katzarkov_kontsevich_2014_dynamical_systems_and_categories}}]\label{definition: mass growth}
Let $\sigma$ be a stability condition on $D^b(X)$. The \emph{mass growth} of a Fourier--Mukai endofunctor $\Phi$ with respect to $\sigma$ is defined by 
\[
h_{\sigma,t}(\Phi)\coloneqq\sup_{0\ne E\in D^b(X)}\limsup_{n\to\infty}\frac{1}{n}\log m_{\sigma,t}(\Phi^nE).
\]
We write $h_\sigma(\Phi)=h_{\sigma,0}(\Phi)$.
\end{dfn}

\begin{thm}[{\cite[Propositions 3.4, 3.11 and Theorem 3.5]{ikeda_2021_mass_growth_of_objects_and_categorical_entropy}}]\label{theorem: mass estimates and generator formula}
Let $\sigma$ be a stability condition on $D^b(X)$ and let $t\in\RB$.
\begin{enumerate}
\item For $0\ne A\in D^b(X)$ and $E\in D^b(X)$, one has
\[
m_{\sigma,t}(E)\leq m_{\sigma,t}(A)\delta_t(A,E).
\]
\item For a classical generator $G$ and a Fourier--Mukai endofunctor $\Phi$, one has
\[
h_{\sigma,t}(\Phi)=\limsup_{n\to\infty}\frac{1}{n}\log m_{\sigma,t}(\Phi^nG)\leq h_t(\Phi).
\]
\item If $\Phi\in\Auteq(X)$ and $\sigma\in\Stab_N(X)$, then $\log\rho(\numg{\Phi})\leq h_\sigma(\Phi)$.
\end{enumerate}
\end{thm}

\subsection{Mass--Hom bounds}\label{subsection: mass Hom bounds}
We recall the mass--Hom bound introduced by Halpern-Leistner and Robotis.

\begin{dfn}[{\cite[Definition 2.7]{preprint_daniel_antoniosalexandros_2025_the_space_of_augmented_stability_conditions}}]\label{definition: mass Hom bound}
A stability condition $\sigma$ on $D^b(X)$ satisfies a \emph{mass--Hom bound} if, for every $A\in D^b(X)$, there is a constant $C(A,\sigma)>0$ such that
\[
\hom(A,E)\leq C(A,\sigma)m_\sigma(E)
\]
for every $E\in D^b(X)$. We set
\[
\MH_X\coloneqq\{\sigma\in\Stab_N(X)\mid\sigma\text{ satisfies a mass--Hom bound}\}.
\]
\end{dfn}


\begin{prop}[{\cite[Proposition 2.13 and Corollary 2.14]{preprint_daniel_antoniosalexandros_2025_the_space_of_augmented_stability_conditions}}]\label{proposition: total Hom mass bound and deformation}
Let $\sigma$ be a stability condition on $D^b(X)$ satisfying a mass--Hom bound. For every $A\in D^b(X)$, there is a constant $C(A,\sigma)>0$ such that
\[
\sum_{i\in\ZZ}\hom^i(A,E)\leq C(A,\sigma)m_\sigma(E)
\]
for every $E\in D^b(X)$. Every stability condition in the same connected component as $\sigma$ also satisfies a mass--Hom bound.
\end{prop}

Thus, $\MH_X$ is a union of connected components. It is also invariant under $\Auteq(X)$. Indeed, for $\sigma\in\MH_X$ and $\Phi\in\Auteq(X)$, one has $\hom(A,E)=\hom(\Phi^{-1}A,\Phi^{-1}E)\leq C(\Phi^{-1}A,\sigma)m_{\Phi\sigma}(E)$ by \eqref{equation: covariance of mass}.

In our setting, the existence of numerical stability conditions satisfying a mass--Hom bound follows from \cite{preprint_chunyi_zhiyu_ziqi_emanuele_alexander_paolo_xiaolei_2026_stability_conditions_and_moduli_spaces_on_projective_families,preprint_yiran_2026_a_remark_on_the_full_support_property}.
\begin{thm} [\cite{preprint_chunyi_zhiyu_ziqi_emanuele_alexander_paolo_xiaolei_2026_stability_conditions_and_moduli_spaces_on_projective_families,preprint_yiran_2026_a_remark_on_the_full_support_property}]\label{theorem: existence of numerical mass Hom stability conditions}
For every smooth projective variety $X$, one has $\MH_X\ne\varnothing$.
\end{thm}

\subsection{Abelian varieties}
\label{subsection: abelian varieties and point objects}
In this subsection, we review basic facts on abelian varieties and fix some notaion; see \cite{book_birkenhake_lange_2004_complex_abelian_varieties} for more detail. 
For an abelian variety $A$, we write $\widehat A=\Pic^0(A)$ for its dual abelian variety and $t_a:A\to A$, $x\mapsto x+a$, for translation by $a\in A$.
For each integer $m>0$, put $A[m]\coloneqq\Ker([m]_A:A\to A)$, where $[m]_A(x)=mx$.
We denote by $\PC_A$ the normalized Poincar\'e line bundle on $A\times\widehat A$.
For $\alpha\in\widehat A$, put $P_\alpha=\PC_A|_{A\times\{\alpha\}}$.

\begin{prop}[{\cite{book_birkenhake_lange_2004_complex_abelian_varieties}}]
\label{proposition: basic facts about abelian varieties}
Let $A$ be an abelian variety and let $Y$ be a smooth projective variety.
Write $\pi_Y:A\times Y\to Y$ for the projection.
\begin{enumerate}
\item For every morphism $f:A\to B$ to an abelian variety $B$, there is a unique homomorphism $u:A\to B$ such that $f=t_{f(0)}\circ u$.
Moreover, $A[m]$ is finite for every integer $m>0$.

\item Cup product identifies $H^k(A,\CB)$ with $\bigwedge^kH^1(A,\CB)$, and the tangent bundle of $A$ is trivial.
Thus, inversion $-\id_A:A\to A$, $x\mapsto-x$, acts by $(-1)^k$ on $H^k(A,\CB)$ and $\td(A)=1$.

\item For $L\in\Pic(A)$, the map $\phi_L:A\to\widehat A$ defined by $a\mapsto t_a^*L\otimes L^{-1}$ is a homomorphism.
It is zero if and only if $L\in\Pic^0(A)$.

\item If $N\in\Pic(A\times Y)$ satisfies $N|_{A\times\{y\}}\in\Pic^0(A)$ for every $y\in Y$, then there are a morphism $g:Y\to\widehat A$ and a line bundle $M\in\Pic(Y)$ such that
\[
N\cong(\id_A\times g)^*\PC_A\otimes\pi_Y^*M.
\]

\item Fix a point $y_0\in Y$, and let $a_Y:Y\to\Alb(Y)$ be the Albanese morphism with $a_Y(y_0)=0$.
For every morphism $f:Y\to B$ to an abelian variety $B$, there is a unique homomorphism $u:\Alb(Y)\to B$ such that
\[
f=t_{f(y_0)}\circ u\circ a_Y.
\]
Moreover, $a_Y^*:\Pic^0(\Alb(Y))\to\Pic^0(Y)$ is an isomorphism.
\end{enumerate}
\end{prop}

Translations and tensoring by elements of $\Pic^0(A)$ act trivially on $N(A)$.

We call an object \emph{simple} if $\End(E)=\CB$ holds.
\begin{dfn}[{\cite[Section 14.9]{book_birkenhake_lange_2004_complex_abelian_varieties}}]\label{definition: semihomogeneous vector bundle}
A vector bundle $E$ on $A$ is \emph{semihomogeneous} if, for every $a\in A$, there exists $\alpha\in\widehat A$ such that $t_a^*E\cong E\otimes P_\alpha$.
\end{dfn}

Every line bundle is simple and semihomogeneous. We will use the following stability result of Fu--Li--Zhao.

\begin{thm}[{\cite[Theorem 2.13 and Corollaries 2.15, 2.16]{fu_li_zhao_2022_stability_manifolds_of_varieties_with_finite_albanese_morphisms}}]\label{theorem: universal stability for finite Albanese morphisms}
\begin{enumerate}
\item If $a_X$ is finite, then all skyscraper sheaves are stable of the same phase with respect to every $\sigma\in\Stab_N(X)$.
\item Every simple semihomogeneous vector bundle on an abelian variety $A$ is stable with respect to every $\sigma\in\Stab_N(A)$.
\end{enumerate}
\end{thm}

\section{Equality of mass growth and categorical entropy}\label{section: equality of mass growth and categorical entropy}

In this section, we prove \cref{theorem in intro: comparison of entropy and mass growth}.
We first extend the total Hom estimate in \cref{proposition: total Hom mass bound and deformation} to arbitrary real weights.

\begin{prop}\label{proposition: weighted total Hom mass bound}
Let $\sigma$ be a stability condition on $D^b(X)$ satisfying a mass--Hom bound.
For every $F\in D^b(X)$ and $t\in\RB$, there is a constant $C(F,\sigma,t)>0$ such that for every $E\in D^b(X)$,
\[
\sum_{i\in\ZZ}\hom^i(F,E)e^{-it}\leq C(F,\sigma,t)m_{\sigma,t}(E).
\]
\end{prop}
\begin{proof}
We may assume that $F\neq0$.
First, we show the statement for a $\sigma$-semistable object $E\in\PC(\phi)$.
Let $a\coloneqq \phi_\sigma^-(F)$, $b\coloneqq \phi_\sigma^+(F\otimes\omega_X[\dim X])$.
Note that if
\[
\hom(F, E[k]) = \hom(E[k], F\otimes\omega_X[\dim X])\neq0,
\]
then $a\leq \phi+k\leq b$ by the orthogonality of slicing, and therefore we have
\[
e^{-kt}=e^{t\phi}e^{-t(\phi+k)} \leq e^{t\phi}\max\{e^{-ta},e^{-tb}\}.
\]

Thus, there exists $C>0$ that depends only on $F$ and $\sigma$ such that
\begin{align*}
\sum_{i\in\ZZ}\hom^i(F, E)e^{-it}
&\le e^{t\phi}\max\{e^{-ta},e^{-tb}\}\sum_{i\in\ZZ}\hom^i(F, E)\\
&\le C\max\{e^{-ta},e^{-tb}\}\,m_{\sigma,t}(E),
\end{align*}
where the second inequality follows from \cref{proposition: total Hom mass bound and deformation}.

For a general $E\in D^b(X)$, take $\sigma$-semistable factors $P_1, \ldots, P_l$ of $E$.
As
\[
\sum_{i\in\ZZ}\hom^i(F, E)e^{-it} \le \sum_{j=1}^{l}\sum_{i\in\ZZ}\hom^i(F, P_j)e^{-it},
\]
the statement holds by the $\sigma$-semistable case.
\end{proof}

\begin{cor}\label{corollary: equality of entropy and mass growth}
Let $X$ be a smooth projective variety, and let $\sigma$ be a stability condition on $D^b(X)$ satisfying a mass--Hom bound.
For every classical generator $G$ of $D^b(X)$ and every $t\in\RB$, there are constants $c(G,\sigma,t),C(G,\sigma,t)>0$ such that
\[
c(G,\sigma,t)m_{\sigma,t}(E)
\leq\delta_t(G,E)
\leq C(G,\sigma,t)m_{\sigma,t}(E)
\]
for every $E\in D^b(X)$.
Consequently, for every Fourier--Mukai endofunctor $\Phi$ of $D^b(X)$, the following limit exists and satisfies
\[
h_t(\Phi)=h_{\sigma,t}(\Phi)
=\lim_{n\to\infty}\frac{1}{n}\log m_{\sigma,t}(\Phi^nG).
\]
\end{cor}
\begin{proof}
Fix $G$ and $t$.
By \cref{theorem: mass estimates and generator formula}, the lower bound holds with $c(G,\sigma,t)=m_{\sigma,t}(G)^{-1}$.
On the other hand, \cref{theorem: categorical entropy and total Hom comparison} gives a constant $a(G,t)>0$ such that
\[
\delta_t(G,E)
\leq\frac{1}{a(G,t)}\sum_{i\in\ZZ}\hom^i(G,E)e^{-it}
\]
for every $E\in D^b(X)$.
Combining this with \cref{proposition: weighted total Hom mass bound} gives the upper bound.

Let $\Phi$ be a Fourier--Mukai endofunctor of $D^b(X)$.
If $\Phi^nG=0$ for some $n$, all subsequent iterates vanish, and both growth rates are $-\infty$ by \cref{theorem: categorical entropy and total Hom comparison,theorem: mass estimates and generator formula}.
Otherwise, applying the two-sided comparison to $E=\Phi^nG$ shows that
$\log\delta_t(G,\Phi^nG)-\log m_{\sigma,t}(\Phi^nG)$ is bounded independently of $n$.
Since $\displaystyle \frac{\log\delta_t(G,\Phi^nG)}{n}$ converges to $h_t(\Phi)$, the limit $\displaystyle\lim_{n\to\infty}\frac{\log m_{\sigma,t}(\Phi^nG)}{n}$ also exists and equals $h_t(\Phi)$.
The generator formula in \cref{theorem: mass estimates and generator formula} identifies this limit with $h_{\sigma,t}(\Phi)$.
\end{proof}

\section{The orbit escape principle}\label{section: orbit escape principle}

We now compare the mass growth of an autoequivalence with its numerical action. Throughout this section, fix a norm on $N(X)_{\RB}$ and the associated operator norm.

\begin{dfn}\label{definition: semistability locus and return times}
Let $\GC=\{G_1,\ldots,G_r\}$ be a finite generating set of $D^b(X)$.
We define
\[
\Sigma_{\GC}\coloneqq
\{\tau\in\Stab_N(X)\mid G_i\text{ is }\tau\text{-semistable for all }i\}
\]
and call it the \emph{$\GC$-semistable locus}.
The locus on which all $G_i$ are stable is denoted by $\Sigma_{\GC}^{\mathrm s}$.
\end{dfn}

The objects in $\GC$ need not have the same phase. We put $G=\bigoplus_{i=1}^rG_i$ when applying the generator formula for mass growth. By the continuity of extremal phases, $\Sigma_{\GC}$ is closed, whereas $\Sigma_{\GC}^{\mathrm s}$ is open. Either locus may be empty.
For every $\sigma\in\Stab_N(X)$ and classical generator $G$, there is a finite generating set $\GC$ such that $\sigma\in\Sigma_\GC$. For example, $\GC=\{G_1, \dots G_l\}$ satisfies the condition, where $G_1, \ldots, G_l$ are the $\sigma$-semistable factors of $G$.

The following estimate follows from the support property and Ikeda's complexity estimate in \cref{theorem: mass estimates and generator formula}.

\begin{lem}\label{lemma: numerical operator norm bounded by mass}
For $\sigma\in\Stab_N(X)$ and a classical generator $G$, there is $c(G,\sigma)>0$ such that, for every $\Phi\in\Auteq(X)$ and every $n\geq0$,
\[
c(G,\sigma)\norm{(\numg{\Phi})^n}
\leq m_\sigma(\Phi^nG).
\]
\end{lem}
To formulate the orbit principle theorem, we introduce a notation $\vtime_{\GC, \sigma}(\Phi)\coloneqq \{n\geq0\mid \Phi^{-n}(\sigma)\in\Sigma_\GC\}$.

\begin{lem}\label{lemma: mass estimate at return times}
Let $\GC=\{G_1,\ldots,G_r\}$ be a finite generating set, put $G=\bigoplus_iG_i$, and let $\sigma\in\Stab_N(X)$. There is $C(\GC,\sigma)>0$ such that
\[
n\in \vtime_{\GC,\sigma}(\Phi)
\quad\Longrightarrow\quad
m_\sigma(\Phi^nG)\leq C(\GC,\sigma)\norm{(\numg{\Phi})^n}.
\]
\end{lem}
\begin{proof}
If $\Phi^{-n}\sigma\in\Sigma_{\GC}$, each $\Phi^nG_i$ is $\sigma$-semistable. 
Therefore, additivity of mass under direct sums gives
\begin{equation*}
m_\sigma(\Phi^nG)=\sum_{i=1}^r\abs{Z_\sigma((\numg{\Phi})^n[G_i])}\leq\norm{Z_\sigma}\norm{(\numg{\Phi})^n}\sum_{i=1}^r\norm{[G_i]}.\qedhere
\end{equation*}
\end{proof}

\begin{thm}\label{theorem: orbit escape principle}
Let $\Phi\in\Auteq(X)$, $\sigma\in\MH_X$, and let $\GC$ be a finite generating set. If $\vtime_{\GC,\sigma}(\Phi)$ is infinite, then
\[
\hcat(\Phi)=h_\sigma(\Phi)=\log\rho(\numg{\Phi}).
\]
Equivalently,
\[
\hcat(\Phi)>\log\rho(\numg{\Phi})
\quad\Longrightarrow\quad
\exists N, \forall n\geq N, \Phi^{-n}\sigma\notin\Sigma_{\GC}.
\]
\end{thm}
\begin{proof}
Put $G=\bigoplus_{G_i\in\GC}G_i$, and choose $n_1<n_2<\cdots$ with $n_j\to\infty$ where $n_j\in\vtime_{\GC, \sigma}(\Phi)$. By \cref{lemma: mass estimate at return times}, we have the inequality $m_\sigma(\Phi^{n_j}G)\leq C(\GC,\sigma)\norm{(\numg{\Phi})^{n_j}}$.
Taking the subsequence $n_j$ and using Gelfand's formula, we obtain
\[
h_\sigma(\Phi)\leq\log\rho(\numg{\Phi}).
\]
The reverse inequality follows from \cref{theorem: numerical lower bound for categorical entropy}, and the equality with mass growth follows from \cref{corollary: equality of entropy and mass growth}. 
\end{proof}

The converse to the first assertion need not hold, as the following example shows.

\begin{eg}\label{example: escape with zero entropy}
Let $X=\PP1$, $\GC=\{\OO,\OO(1)\}$, and $\Phi=\tensfunc_{\OO(1)}$.
The set $\GC$ is a finite generating set by \cref{theorem: existence of generator} and $\hcat(\Phi)=\log\rho(\numg{\Phi})=0$ by \cref{lemma: categorical entropy of tensor functors}.

The Ext-exceptional pair $(\OO[1],\OO(1))$ generates a finite-length heart $\AC=\langle\OO[1],\OO(1)\rangle_{\mathrm{ex}}$.
By \cite[Lemmas 3.14 and 3.16]{macri_2007_stability_conditions_on_curves}, the central charge
\[
Z(\OO[1])=i,\quad
Z(\OO(1))=-1+i
\]
defines a numerical stability condition $\sigma$ with heart $\AC$, whose two simple objects have phases $1/2$ and $3/4$, respectively.
Moreover, $\sigma$ satisfies a mass--Hom bound by \cite[Example 2.9]{preprint_daniel_antoniosalexandros_2025_the_space_of_augmented_stability_conditions}.
Therefore, $\sigma\in\MH_X\cap\Sigma_{\GC}^{\mathrm s}$.

For every integer $n\geq1$, the evaluation sequence of $\OO(n)$ tensored by $\OO(1)$ is
\[
0\longrightarrow\OO^{\oplus n}\longrightarrow\OO(1)^{\oplus(n+1)}
\longrightarrow\OO(n+1)\longrightarrow0.
\]
Rotating the corresponding exact triangle gives a short exact sequence in $\AC$,
\[
0\longrightarrow\OO(1)^{\oplus(n+1)}\longrightarrow\OO(n+1)
\longrightarrow\OO[1]^{\oplus n}\longrightarrow0.
\]
Since $Z(\OO(n+1))=-(n+1)+(2n+1)i$, we have
\[
\arg Z(\OO(n+1))<\frac{3\pi}{4}=\arg Z(\OO(1)).
\]
Thus, $\OO(1)^{\oplus(n+1)}$ destabilizes $\OO(n+1)$ for every $n\geq1$.
Since $\Phi^n(\OO(1))=\OO(n+1)$, it follows that $\Phi^{-n}\sigma\notin\Sigma_{\GC}$ for every $n\geq1$.
Together with $\sigma\in\Sigma_{\GC}^{\mathrm s}$, this gives
\[
\vtime_{\GC,\sigma}(\Phi)=\{0\}.
\]
\end{eg}

\begin{rem}[Compatible triples]\label{remark: compatible triples and orbit escape}
Barbacovi--Kim \cite{barbacovi_kim_2023_on_gromovyomdin_type_theorems_and_a_categorical_interpretation_of_holomorphicity} call a triple
$(\Phi,\sigma,g)\in\Auteq(X)\times\Stab_N(X)\times\GLt$
\emph{compatible} if $\Phi\sigma=\sigma\cdot g$, where the right-hand side is the standard action of the universal cover of $\GL^+(2,\RB)$.
This action commutes with autoequivalences and preserves semistable objects.
Choose a finite generating set $\GC$ of $\sigma$-semistable objects; the Harder--Narasimhan factors of any classical generator form such a set.
Compatibility then gives
\[
\{\Phi^n\sigma\mid n\in\ZZ\}
\subseteq\sigma\cdot\GLt\subseteq\Sigma_{\GC}.
\]
Thus, compatible triples give nonescaping orbits in the present setting.
If $\sigma\in\MH_X$, \cref{theorem: orbit escape principle} yields
$\hcat(\Phi)=h_\sigma(\Phi)=\log\rho(\numg{\Phi})$.
In contrast to compatibility, the orbit criterion only requires infinitely many returns to the semistability locus of a finite generating set.
\end{rem}

\section{The Gromov--Yomdin type equality for abelian varieties}\label{section: abelian varieties}

\begin{prop}\label{proposition: universally stable generator on abelian varieties}
Let $A$ be an abelian variety of dimension $g$, $L$ be a very ample line bundle, $\GC=\{\OO_A,L,\ldots,L^{\otimes g}\}$ be a finite generating set of $D^b(A)$.
Then, 
\[
\Sigma_{\GC}^{\mathrm s}=\Sigma_{\GC}=\Stab_N(A).
\]
\end{prop}
\begin{proof}
Since every $L^{\otimes j}$ is a simple semihomogeneous vector bundle, the assertion follows from \cref{theorem: universal stability for finite Albanese morphisms}.
\end{proof}

\begin{thm}\label{theorem: Gromov Yomdin equality for abelian varieties}
For every abelian variety $A$, every $\Phi\in\Auteq(A)$, and every $\sigma\in\Stab_N(A)$, we have
\[
\hcat(\Phi)=h_\sigma(\Phi)=\log\rho(\numg{\Phi}).
\]
\end{thm}
\begin{proof}
Choose $\GC$ as in
\cref{proposition: universally stable generator on abelian varieties}
and $\sigma_0\in\MH_A$ by
\cref{theorem: existence of numerical mass Hom stability conditions}.
Since $\Phi^{-n}\sigma_0\in\Sigma_{\GC}$ for every $n\geq0$,
\cref{theorem: orbit escape principle} gives
\[
    \hcat(\Phi)=\log\rho(\numg{\Phi}).
\]
For every $\sigma\in\Stab_N(A)$,
\cref{theorem: mass estimates and generator formula} gives
\[
    \log\rho(\numg{\Phi})\leq h_\sigma(\Phi)\leq\hcat(\Phi),
\]
which proves the assertion.
\end{proof}

\subsection{The cohomological spectral radius}\label{subsection: cohomological spectral radius}
For an autoequivalence $\Phi$ of an abelian variety, put $\rho_H(\Phi)=\rho(\Phi^H\colon H^*(A,\CB) \to H^*(A,\CB))$.
In general, $\log\rho(\numg{\Phi})\leq \log\rho_H(\Phi)$ holds, whereas the reverse inequality does not necessarily hold.
In this subsection, we show that the categorical entropy agrees with both $\log\rho(\numg{\Phi})$ and $\log\rho_H(\Phi)$.

\begin{lem}\label{lemma: entropy of external product with identity}
Let $X,Y$ be smooth projective varieties and $\Psi\in\Auteq(Y)$. Then
\[
\hcat(\id_{D^b(X)}\boxtimes\Psi)=\hcat(\Psi).
\]
\end{lem}
\begin{proof}
If $G_X,G_Y$ are classical generators, then $G_X\boxtimes G_Y$ is a classical generator of $D^b(X\times Y)$. The K\"unneth formula gives
\[
\sum_{i\in\ZZ}\hom^i(G_X\boxtimes G_Y,G_X\boxtimes\Psi^nG_Y)
=\left(\sum_{i\in\ZZ}\hom^i(G_X,G_X)\right)\left(\sum_{i\in\ZZ}\hom^i(G_Y,\Psi^nG_Y)\right).
\]
The assertion follows from \cref{theorem: categorical entropy and total Hom comparison}.
\end{proof}

Let $X$ be a smooth projective variety.
For an endomorphism $T$ of $H^*(X,\CB)$ preserving
$H^{\mathrm{ev}}(X,\CB)$ and $H^{\mathrm{odd}}(X,\CB)$, its \emph{supertrace} is
defined by
\[
\sTr(T)
=
\Tr(T|_{H^{\mathrm{ev}}(X,\CB)})
-
\Tr(T|_{H^{\mathrm{odd}}(X,\CB)}).
\]
Note that the cohomological action of a Fourier--Mukai functor preserves this decomposition.
For any $K\in D^b(X\times X)$, the cohomological Lefschetz formula
\cite[Theorem 1.2 and Remark 4.5]{lunts_2012_lefschetz_fixed_point_theorems_for_fouriermukai_functors_and_dg_algebras}
gives
\[
\sTr(\Phi_K^H)
=
\int_X\Delta_X^*\ch(K)\td(X),
\]
where $\Delta_X:X\hookrightarrow X\times X$ is the diagonal embedding.

\begin{prop}\label{proposition: numerical and cohomological entropy on abelian varieties}
For every autoequivalence $\Phi$ of an abelian variety $A$,
\[
\hcat(\Phi)=\log\rho(\numg{\Phi})=\log\rho_H(\Phi).
\]
\end{prop}
\begin{proof}
Let $g=\dim A$, let $\iota=-\id_A$, and let $\Gamma_\iota\subset A\times A$ be its graph.
Denote the diagonal by $\Delta$, and write $K_n$ for the Fourier--Mukai kernel of $\Phi^n$.

Since $\iota^*$ acts as the identity on $H^{\mathrm{ev}}(A,\CB)$
and as $-\id$ on $H^{\mathrm{odd}}(A,\CB)$, we have
\[
\sTr((\iota^*\circ\Phi^n)^H)=\Tr((\Phi^H)^n).
\]
Since $\iota^*\circ\Phi^n$ has Fourier-Mukai kernel $(\id_A\times\iota)^*K_n$, $(\id_A\times\iota)(\Delta)=\Gamma_\iota$, and $\td(A)=1$ by \cref{proposition: basic facts about abelian varieties}, the Lefschetz formula above and the projection formula give
\begin{align*}
\Tr((\Phi^H)^n)
&=\sTr((\iota^*\circ\Phi^n)^H)\\
&=\int_{A}\Delta_{A}^*\ch((\id_A\times\iota)^*K_n)
=\int_{A\times A}[\Gamma_\iota]\ch(K_n).
\end{align*}
The normal bundle of $\Gamma_\iota$ is trivial of rank $g$.
Thus, duality for the graph embedding gives
the derived dual $\OO_{\Gamma_\iota}^{\vee}\cong\OO_{\Gamma_\iota}[-g]$,
while Grothendieck--Riemann--Roch gives
$\ch(\OO_{\Gamma_\iota})=[\Gamma_\iota]$.
Consequently,
$\ch(\OO_{\Gamma_\iota}^{\vee})=(-1)^g[\Gamma_\iota]$,
and Riemann--Roch yields
\[\Tr((\Phi^H)^n)
=(-1)^g\chi_{A\times A}(\OO_{\Gamma_\iota},K_n).\]

Set $\Phi_A=\id_{D^b(A)}\boxtimes\Phi$, and choose a classical
generator $H$ of $D^b(A\times A)$.
Since $K_n\cong\Phi_A^n(\OO_\Delta)$ and
$\OO_{\Gamma_\iota}$ and $\OO_\Delta$ are fixed objects,
\cref{lemma: fixed objects and total Hom growth}
gives a constant $C>0$, which is independent of $n$, such that
\begin{align*}
\abs{\Tr((\Phi^H)^n)}
&\leq
\sum_{i\in\ZZ}
\hom^i(\OO_{\Gamma_\iota},\Phi_A^n(\OO_\Delta))\\
&\leq
C\sum_{i\in\ZZ}\hom^i(H,\Phi_A^nH).
\end{align*}
The spectral radius of the linear operator $\Phi^H$ satisfies
$\rho_H(\Phi)
=\limsup_{n\to\infty}\abs{\Tr((\Phi^H)^n)}^{1/n}$.
Taking exponential growth rates in the preceding inequality and
applying
\cref{theorem: categorical entropy and total Hom comparison,lemma: entropy of external product with identity},
we obtain
\[
\log\rho_H(\Phi)
\leq\hcat(\Phi_A)
=\hcat(\Phi).
\]
On the other hand,
\cref{theorem: Gromov Yomdin equality for abelian varieties}
and \eqref{equation: numerical and cohomological spectral radii}
give $\hcat(\Phi)=\log\rho(\numg{\Phi})\leq\log\rho_H(\Phi)$.
This proves the assertion.
\end{proof}

\section{Applications to finite Albanese morphisms}\label{section: finite Albanese morphisms}

We now prove the Gromov--Yomdin property for varieties with finite Albanese morphisms (\cref{theorem in intro: Gromov Yomdin equality for finite Albanese morphisms}).

\begin{thm}\label{theorem: Gromov Yomdin equality for finite Albanese morphisms}
Let $X$ be a smooth projective variety with finite Albanese morphism. For every $\Phi\in\Auteq(X)$ and every $\sigma\in\Stab_N(X)$,
\[
\hcat(\Phi)=h_\sigma(\Phi)=\log\rho(\numg{\Phi}).
\]
\end{thm}

\subsection{Lifting a power to the product cover}\label{subsection: lifting to the product cover}

For a finite \'{e}tale Galois cover $p:\widetilde X\to X$, its group of deck transformations is denoted by $\Gal(p)$.
We use the following form of Kawamata's structure theorem \cite[Theorem 13]{kawamata_1981_characterization_of_abelian_varieties}.

\begin{thm}[\cite{kawamata_1981_characterization_of_abelian_varieties} and {\cite[Theorem 3.2]{jiang_2011_an_effective_version_of_a_theorem_of_kawamata_on_the_albanese_map}}]
\label{theorem: Kawamata product cover}
If $a_X$ is finite, then there is a finite \'{e}tale Galois cover
\[
p:P=B\times Y\longrightarrow X
\]
such that $B$ is an abelian variety of dimension $\dim X-\kappa(X)$, $Y$ is a smooth variety of general type with finite Albanese morphism and dimension $\kappa(X)$, and $\Gal(p)$ is a finite abelian group acting diagonally, by translations on $B$.

The cover can be chosen as a connected component of a base change of $a_X$ by an isogeny.
\end{thm}

\begin{dfn}\label{definition: connected automorphism group}
The \emph{connected automorphism group} $\Aut^0(X)$ is the identity component of the automorphism group scheme of $X$. It is a connected algebraic group since we work over $\CB$. Similarly, $\Pic^0(X)$ is the identity component of the Picard scheme; it parametrizes algebraically trivial line bundles. We put
\[
R_X\coloneqq\Aut^0(X)\times\Pic^0(X).
\]
\end{dfn}

An equivalence $\Phi:D^b(X)\to D^b(Y)$ induces an isomorphism of algebraic groups $R_X\cong R_Y$ (\cite[Section 2]{caucci_lombardi_2025_irregular_fibrations_of_derived_equivalent_varieties}).

\begin{dfn}\label{definition: Rouquier isomorphism and stability}
The \emph{Rouquier isomorphism} associated with an equivalence $\Phi\colon D^b(X)\to D^b(Y)$ is the isomorphism
$\varphi_\Phi:R_X\to R_Y$ characterized by the condition that $\varphi_\Phi(f,L)=(g,M)$ if and only if
\[
\Phi\circ f_*\circ\tensfunc_L\cong g_*\circ\tensfunc_M\circ\Phi.
\]
For an abelian subvariety $V\subset\Pic^0(X)$, we say that $V$ is \emph{Rouquier-stable with respect to $\Phi$} if
\[
\varphi_\Phi(\{\id_X\}\times V)
\subset\{\id_Y\}\times\Pic^0(Y).
\]
If $X=Y$, we say that $\varphi_\Phi$ \emph{preserves $V$} if $\varphi_\Phi(\{\id_X\}\times V)=\{\id_X\}\times V$.
\end{dfn}

Dualizing the inclusion $V\hookrightarrow\Pic^0(X)$ gives a quotient $\Alb(X)\to\widehat V$, using the identification $\widehat{\Pic^0(X)}\cong\Alb(X)$ (see \cite[Propositions 2.4.2 and 11.11.6]{book_birkenhake_lange_2004_complex_abelian_varieties}).
The composition $X\xrightarrow{a_X}\Alb(X)\to\widehat V$ is called the morphism associated with $V$.

\begin{thm}[{\cite[Proposition 4.3 and Remark 4.4]{krug_sosna_2015_equivalences_of_equivariant_derived_categories}}]
\label{theorem: lifting equivalences to cyclic covers}
Let $q:\widetilde X\to X$ be a connected cyclic \'{e}tale cover defined by a torsion line bundle $L$.
If $\Psi\in\Auteq(X)$ satisfies $\Psi\circ \tensfunc_L\cong \tensfunc_L\circ\Psi$, then it has a lift $\widetilde\Psi\in\Auteq(\widetilde X)$ satisfying
$\widetilde\Psi \circ q^*\cong q^*\circ\Psi$ and $q_*\circ\widetilde\Psi\cong\Psi\circ q_*$.
\end{thm}

\begin{lem}[{cf.\ \cite[Proposition 4.2]{yoshida_2026_a_note_on_categorical_entropy_of_bielliptic_surfaces_and_enriques_surfaces}}]\label{lemma: categorical entropy of lifts}
Let $p:\widetilde X\to X$ be a finite \'{e}tale Galois cover. Suppose that $\widetilde\Phi\in\Auteq(\widetilde X)$ and $\Phi\in\Auteq(X)$ satisfy $\widetilde\Phi\circ p^*\cong p^*\circ\Phi$ and $p_*\circ\widetilde\Phi\cong\Phi\circ p_*$. Then $h_t(\widetilde\Phi)=h_t(\Phi)$ for every $t\in\RB$.
\end{lem}
\begin{proof}
Let $G$ be a classical generator of $D^b(X)$.
For every $E\in D^b(\widetilde X)$, the object $p^*p_*E\cong\bigoplus_{\gamma\in\Gal(p)}\gamma_*E$ belongs to $\thcl{p^*G}$ and contains $E$ as a direct summand.
Thus, $p^*G$ is a classical generator of $D^b(\widetilde X)$.
The trace splitting makes $G$ a direct summand of $p_*p^*G$, so the latter is a classical generator of $D^b(X)$.
Adjunction and the assumed isomorphisms of functors give
\[
\sum_{i\in\ZZ}\hom^i(p^*G,\widetilde\Phi^np^*G)e^{-it}
=\sum_{i\in\ZZ}\hom^i(G,\Phi^n(p_*p^*G))e^{-it}.
\]
The assertion follows from \cref{theorem: categorical entropy and total Hom comparison}.
\end{proof}

For a connected finite \'{e}tale Galois cover $p:\widetilde X\to X$ with group $\Gamma=\Gal(p)$, there is a natural isomorphism
\[
\Hom(\Gamma,\CB^*)\cong\Ker\bigl(p^*:\Pic(X)\to\Pic(\widetilde X)\bigr), 
\]
see \cite[Section 4]{krug_sosna_2015_equivalences_of_equivariant_derived_categories}.
Write $L_\chi$ for the line bundle corresponding to a character $\chi$.
Its order equals that of $\chi$, and pullback from $X$ to $\widetilde X/\Gamma'$ corresponds to restriction from $\Gamma$ to $\Gamma'$ for every subgroup $\Gamma'\subset\Gamma$.
If $\Gamma$ is cyclic and $\chi$ is faithful, then $p$ is the cyclic cover defined by $L_\chi$.

\begin{prop}\label{proposition: lifting a power to the product cover}
Let $X$ have finite Albanese morphism, and let $p:\widetilde X\to X$ be a connected finite \'{e}tale Galois cover with abelian group $\Gal(p)$.
Assume that
\[
\Ker\bigl(p^*:\Pic(X)\to\Pic(\widetilde X)\bigr)
\subset\Pic^0(X).
\]
For every $\Phi\in\Auteq(X)$, there are an integer $m>0$ and $\widetilde\Phi\in\Auteq(\widetilde X)$ such that $\widetilde\Phi\circ p^*\cong p^*\circ\Phi^m$ and $p_*\circ\widetilde\Phi\cong\Phi^m\circ p_*$.
Moreover, $h_t(\widetilde\Phi)=mh_t(\Phi)$ for all $t\in\RB$.
\end{prop}

\begin{proof}
We first suppose that $p$ is cyclic, defined by a line bundle
$L\in\Pic^0(X)$ of order $d$.
Since $X$ has finite Albanese morphism, it is not uniruled.
Thus, $R_X$ is an abelian variety by
\cite[Proposition 7.1.4]{brion_2017_some_structure_theorems_for_algebraic_groups}.
The Rouquier automorphism $\varphi_\Phi$ permutes the finite group
$R_X[d]$, so some positive power fixes $(\id_X,L)$.
Thus, for some $m>0$,
\[
    \Phi^m\circ\tensfunc_L\cong\tensfunc_L\circ\Phi^m.
\]
The argument of
\cite[Theorem 4.5]{bridgeland_maciocia_2017_fouriermukai_transforms_for_quotient_varieties} gives a lift
$\widetilde\Phi\in\Auteq(\widetilde X)$ satisfying
\[
    \widetilde\Phi\circ p^*\cong p^*\circ\Phi^m,
    \qquad
    p_*\circ\widetilde\Phi\cong\Phi^m\circ p_*.
\]

Next, we show the general cases.
Since $\Gal(p)$ is abelian, repeatedly choosing a maximal proper subgroup
gives a chain $\Gal(p)=\Gamma_0\supset\cdots\supset\Gamma_s=\{1\}$
with cyclic quotients of prime order.
Setting $X_i=\widetilde X/\Gamma_i$ gives a tower
\[
    \widetilde X=X_s\xrightarrow{q_s}X_{s-1}\longrightarrow\cdots
    \longrightarrow X_1\xrightarrow{q_1}X_0=X
\]
of connected cyclic \'etale covers with
$\Gal(q_i)=\Gamma_{i-1}/\Gamma_i$.
Let $d_i=\deg q_i$, and let $L_i$ be the line bundle on $X_{i-1}$
corresponding to a faithful character of $\Gamma_{i-1}/\Gamma_i$.
Since $\Gamma$ is abelian, this character, viewed as a character
of $\Gamma_{i-1}$, extends to $\Gamma$.
By the character correspondence above, $L_i$ is the pullback
along $X_{i-1}\to X$ of a line bundle trivialized by $p$.
Hence $L_i\in\Pic^0(X_{i-1})$ and $\ord(L_i)=d_i$.
Moreover, the composite $X_i\to X\xrightarrow{a_X}\Alb(X)$ is finite,
so each $X_i$ has finite Albanese morphism.

Applying the cyclic case successively, starting with $\Psi_0=\Phi$,
we choose integers $k_i>0$ and autoequivalences $\Psi_i\in\Auteq(X_i)$
such that $\Psi_i$ lifts $\Psi_{i-1}^{k_i}$ along $q_i$.
The lifting isomorphisms are preserved under taking powers
and compose along the tower.
Thus, $\widetilde\Phi=\Psi_s$ gives the required lift of $\Phi^m$,
where $m=\prod_{i=1}^s k_i$.
Finally,
\[
    h_t(\widetilde\Phi)=h_t(\Phi^m)=mh_t(\Phi)
\]
for every $t\in\RB$ by
\cref{lemma: categorical entropy of lifts,proposition: basic properties of categorical entropy}.
\end{proof}

Suppose that $p:\widetilde X\to X$ is a connected component of the base change of an isogeny $v:A'\to\Alb(X)$ along $a_X$.
Identify $\Gal(p)$ with the stabilizer of this component in $\Ker(v)$.
Since restriction of characters from $\Ker(v)$ to $\Gal(p)$ is surjective, the preceding correspondence and \cite[Proposition 2.4.3]{book_birkenhake_lange_2004_complex_abelian_varieties} give
\[
\Ker(p^*)=a_X^*\Ker(v^*)\subset\Pic^0(X),
\]
where both kernels are taken on Picard groups.
Thus, \cref{proposition: lifting a power to the product cover} applies to the cover in \cref{theorem: Kawamata product cover}, and also to every isogeny of abelian varieties.

\subsection{A normal form over the general-type factor}\label{subsection: relative normal form}
We next describe a normal form for autoequivalences on the \'etale cover.

\begin{lem}[{\cite[Lemma 2.12]{fu_li_zhao_2022_stability_manifolds_of_varieties_with_finite_albanese_morphisms}}]\label{lemma: Picard invariant objects have finite support}
Let $A$ be an abelian variety and $E\in D^b(A)$. If $E\otimes L\cong E$ for all $L\in\Pic^0(A)$, then $E$ has finite support.
\end{lem}

\begin{prop}[{\cite[Proposition 1.12]{ana_darío_carlos_2013_relative_fouriermukai_transforms_for_weierstrass_fibrations_abelian_schemes_and_fano_fibrations}}]
\label{proposition: point preserving equivalences are standard}
Let $\Phi:D^b(X)\to D^b(Y)$ be an equivalence. Suppose that, for every closed point $x\in X$, there are a closed point $y_x\in Y$ and an integer $t_x$ with
$\Phi(\OO_x)\cong\OO_{y_x}[t_x]$.
Then the integers $t_x$ have a common value $t$, and there are an isomorphism $f:X\to Y$ and a line bundle $L\in\Pic(X)$ such that
$\Phi\cong f_*\circ\tensfunc_L[t]$.
\end{prop}

\begin{lem}\label{lemma: Picard preserving equivalences are standard}
Let $X$ have finite Albanese morphism.
If the Rouquier automorphism of $\Psi\in\Auteq(X)$ preserves $\Pic^0(X)$, then $\Psi$ is standard.
\end{lem}
\begin{proof}
For a closed point $x\in X$, put $E_x=\Psi(\OO_x)$.
For every $L\in\Pic^0(X)$, the preservation of $\Pic^0(X)$ gives an $M\in\Pic^0(X)$ such that $\varphi_\Psi(\id_X,M)=(\id_X,L)$.
By \cref{definition: Rouquier isomorphism and stability}, this means that $\Psi\circ\tensfunc_M\cong\tensfunc_L\circ\Psi$.
Applying this isomorphism to $\OO_x$ yields
\[
E_x\otimes L\cong\Psi(\OO_x\otimes M)\cong E_x.
\]
By the projection formula, $a_{X*}E_x$ is invariant under tensoring by every element of $\Pic^0(\Alb(X))$. Therefore, it has finite support by \cref{lemma: Picard invariant objects have finite support}. Since $a_X$ is finite, so does $E_x$. Moreover, equivalence gives $\End(E_x)=\CB$ and $\hom^j(E_x,E_x)=0$ for $j<0$.
By \cite[Lemma 4.5]{book_huybrechts_2006_fouriermukai_transforms_in_algebraic_geometry}, $E_x$ is a shift of a skyscraper sheaf. Therefore, \cref{proposition: point preserving equivalences are standard} proves the assertion.
\end{proof}

For an abelian variety $B$ and a variety $Y$ of general type, let $\pi_B$ and $\pi_Y$ denote the projections from $P=B\times Y$.
The projection $\pi_Y$ is a model of the Iitaka fibration of $P$, and $\pi_Y^*\Pic^0(Y)$ corresponds to the Albanese--Iitaka morphism $a_Y\circ\pi_Y$.
If $a_Y$ is finite, then $P\xrightarrow{\pi_Y}Y\xrightarrow{a_Y}\Alb(Y)$ is its Stein factorization, since $\pi_Y$ has connected fibres and $Y$ is normal.
Applying the results of Caucci--Lombardi \cite{caucci_lombardi_2025_irregular_fibrations_of_derived_equivalent_varieties} to this fibration gives the following theorem.
The compatibility with fibre inclusions follows from flat base change for the relative kernel.

\begin{thm}[{\cite[Sections 2.1.3 and 2.1.6, Theorem 3.1 and Claim 3.3]{caucci_lombardi_2025_irregular_fibrations_of_derived_equivalent_varieties}}]
\label{theorem: relative equivalences over general type factor}
Let $P=B\times Y$, where $B$ is abelian and $Y$ is of general type with finite Albanese morphism, and let $\KC\in D^b(P\times P)$ be the Fourier--Mukai kernel of $\Phi\in\Auteq(P)$.
Write $P_y=B\times\{y\}$ and $i_y:P_y\hookrightarrow P$ for each $y\in Y$.
Then the following hold.
\begin{enumerate}
\item The Rouquier isomorphism $\varphi_\Phi$ preserves $\pi_Y^*\Pic^0(Y)$.
\item There is an automorphism $\psi:Y\to Y$ such that
\[
\Supp(\KC)\subset (\pi_Y\times\pi_Y)^{-1}(\Gamma_\psi),
\]
where $\Gamma_\psi=\{(y,\psi(y))\mid y\in Y\}$ and the inclusion is set-theoretic.
\item There are a dense open subset $U\subset Y$ and an object $\KC_U\in D^b(Z_U)$, where
\[
Z_U=(\pi_Y\times\pi_Y)^{-1}(\Gamma_{\psi|_U})\subset P_U\times P_{\psi(U)},
\]
with $P_U=B\times U$, such that $\KC|_{P_U\times P_{\psi(U)}}\cong j_{U*}\KC_U$ for the closed immersion $j_U:Z_U\hookrightarrow P_U\times P_{\psi(U)}$.
For $y\in U$, let $\KC_y$ be the derived pullback of $\KC_U$ along $P_y\times P_{\psi(y)}\hookrightarrow Z_U$.
Then $\Phi_{\KC_y}:D^b(P_y)\to D^b(P_{\psi(y)})$ is an equivalence and
\[
\Phi\circ i_{y*}\cong i_{\psi(y)*}\circ\Phi_{\KC_y}.
\]
\item The automorphism $\psi$ has finite order.
\end{enumerate}
\end{thm}

For morphisms $f:Y\to B$ and $g:Y\to\widehat B$, let $\tau_f:P\to P$ be the automorphism defined by $\tau_f(x,y)=(x+f(y),y)$, and put $\PC_g=(\id_B\times g)^*\PC_B$.

\begin{prop}\label{proposition: relative normal form}
Let $P=B\times Y$, where $B$ is abelian and $Y$ is of general type with finite Albanese morphism.
For every $\widetilde\Phi\in\Auteq(P)$, there are an integer $r>0$, an autoequivalence $\Psi\in\Auteq(B)$, a line bundle $M\in\Pic(Y)$, an integer $s$, and morphisms $f:Y\to B$ and $g:Y\to\widehat B$ such that
\begin{equation}\label{equation: relative normal form}
\widetilde\Phi^r\cong\tensfunc_{\pi_Y^*M}[s]\circ\Phi_0,
\end{equation}
where $\Phi_0=(\tau_f)_*\circ\tensfunc_{\PC_g}\circ(\Psi\boxtimes\id_{D^b(Y)})$.
\end{prop}
\begin{proof}
Let $r$ be the order of the automorphism of $Y$ induced by $\widetilde\Phi$ in \cref{theorem: relative equivalences over general type factor}.
Applying the same theorem to $\widetilde\Phi^r$, choose a fibre inclusion $i:B\hookrightarrow P$ over $y_0\in Y$ and $\Psi\in\Auteq(B)$ such that $\widetilde\Phi^r\circ i_*\cong i_*\circ\Psi$.
For $\Phi'=\widetilde\Phi^r\circ(\Psi^{-1}\boxtimes\id_{D^b(Y)})$, we have
\begin{equation}\label{equation: normalized transform on a fibre}
\Phi'\circ i_*\cong i_*.
\end{equation}

We first show that $\Phi'$ is standard.
Since $Y$ is of general type, $\Aut(Y)$ is finite and $H^0(Y,T_Y)=0$.
The K\"unneth decomposition gives $h^0(P,T_P)=h^0(B,T_B)=\dim B$.
Thus, $\Aut^0(P)$ is the subgroup of translations on the first factor by \cite[Theorem 7.1.1]{brion_2017_some_structure_theorems_for_algebraic_groups}.
For $L\in\Pic^0(P)$, write $\varphi_{\Phi'}(\id_P,L)=(t_c\times\id_Y,L')$ with $c\in B$ and $L'\in\Pic^0(P)$.
By \cref{definition: Rouquier isomorphism and stability},
\[
\Phi'\circ\tensfunc_L\cong(t_c\times\id_Y)_*\circ\tensfunc_{L'}\circ\Phi'.
\]
Set $z=(0,y_0)$.
Since \eqref{equation: normalized transform on a fibre} gives $\Phi'(\OO_z)\cong\OO_z$, applying this isomorphism of functors to $\OO_z$ yields
\begin{align*}
\OO_z&\cong\Phi'(\OO_z\otimes L)\\
&\cong(t_c\times\id_Y)_*(\Phi'(\OO_z)\otimes L')
\cong\OO_{(c,y_0)}.
\end{align*}
Isomorphic coherent sheaves have the same support, so $(0,y_0)=(c,y_0)$ and $c=0$.
Hence $\varphi_{\Phi'}(\{\id_P\}\times\Pic^0(P))\subset\{\id_P\}\times\Pic^0(P)$.
The image is a closed connected subgroup of the same dimension as $\Pic^0(P)$, so equality holds.
Since $P$ has finite Albanese morphism, \cref{lemma: Picard preserving equivalences are standard} gives $\Phi'\cong h_*\circ\tensfunc_N[s]$ for some $h\in\Aut(P)$, $N\in\Pic(P)$, and $s\in\ZZ$.

The support assertion in \cref{theorem: relative equivalences over general type factor} implies that $\widetilde\Phi^r$ preserves objects supported on each fibre of $\pi_Y$.
The functor $\Psi^{-1}\boxtimes\id_{D^b(Y)}$ has the same property, and therefore so does $\Phi'$.
Applying this to skyscraper sheaves gives $\pi_Y\circ h=\pi_Y$.
Moreover, \eqref{equation: normalized transform on a fibre} implies that $h$ fixes $B\times\{y_0\}$ pointwise.
Define $f:Y\to B$ by $f(y)=\pi_B(h(0,y))$.
For each $y\in Y$, the automorphism $u_y(x)=\pi_B(h(x,y))-f(y)$ of $B$ fixes the origin, and hence is a group automorphism by \cref{proposition: basic facts about abelian varieties}.
Rigidity of homomorphisms of abelian varieties shows that $u_y$ is independent of $y$.
Since $u_{y_0}=\id_B$, we obtain $h=\tau_f$.

Applying \eqref{equation: normalized transform on a fibre} to $\OO_B$ gives $i_*(N|_{B\times\{y_0\}})[s]\cong i_*\OO_B$.
Taking cohomology sheaves gives $s=0$, and the full faithfulness of pushforward on coherent sheaves gives $N|_{B\times\{y_0\}}\cong\OO_B$.
Since $Y$ is connected, all fibre restrictions of $N$ belong to $\Pic^0(B)$.
By \cref{proposition: basic facts about abelian varieties}, there are a morphism $g:Y\to\widehat B$ and a line bundle $M\in\Pic(Y)$ such that $N\cong\PC_g\otimes\pi_Y^*M$.
Substituting $h=\tau_f$ and this expression for $N$, and commuting $\tensfunc_{\pi_Y^*M}$ with $(\tau_f)_*$, proves \eqref{equation: relative normal form}.
\end{proof}

\subsection{Proof of the GY property for finite Albanese morphisms}\label{subsection: proof of Gromov Yomdin property for finite Albanese morphisms}

\begin{setup}\label[setup]{setup: abelian extension of a relative autoequivalence}
Let $P=B\times Y$ be as in \cref{proposition: relative normal form}, and let $\widetilde\Phi\in\Auteq(P)$.
Fix $r,\Psi,M,s,f,g$, and $\Phi_0$ satisfying \eqref{equation: relative normal form}.
Put $C=\Alb(Y)$ and $A'=B\times C$, and let $q=\id_B\times a_Y:P\to A'$.
The morphism $q$ is finite.
By \cref{proposition: basic facts about abelian varieties}, there are morphisms $\bar f:C\to B$ and $\bar g:C\to\widehat B$ satisfying $f=\bar f\circ a_Y$ and $g=\bar g\circ a_Y$.
Moreover, $\bar f-\bar f(0)$ and $\bar g-\bar g(0)$ are homomorphisms.
Using the same notation for fibrewise translations and Poincar\'e bundles on $A'$, define
\[
\widehat\Phi_0=(\tau_{\bar f})_*\circ\tensfunc_{\PC_{\bar g}}\circ(\Psi\boxtimes\id_{D^b(C)}).
\]
\end{setup}

\begin{prop}\label{proposition: entropy upper bound from abelian factor}
In \cref{setup: abelian extension of a relative autoequivalence}, one has
\[
r\hcat(\widetilde\Phi)=\hcat(\Phi_0)\leq\hcat(\widehat\Phi_0).
\]
\end{prop}
\begin{proof}
The functor $\Phi_0$ is $Y$-linear and commutes with $\tensfunc_{\pi_Y^*M}$.
This tensor functor and its inverse have zero categorical entropy by \cref{lemma: categorical entropy of tensor functors}.
Applying subadditivity in both directions and using \cref{proposition: basic properties of categorical entropy}, we obtain
\[
r\hcat(\widetilde\Phi)=\hcat(\widetilde\Phi^r)=\hcat(\Phi_0).
\]

We next verify the isomorphisms
\begin{equation}\label{equation: intertwining with abelian extension}
q_*\circ\Phi_0\cong\widehat\Phi_0\circ q_*, \text{ and }
\Phi_0\circ q^*\cong q^*\circ\widehat\Phi_0.
\end{equation}
The identities $q\circ\tau_f=\tau_{\bar f}\circ q$ and $\PC_g=q^*\PC_{\bar g}$ give the compatibility of $q_*$ with the fibrewise translations and tensor functors.
For $\Psi\boxtimes\id$, this compatibility follows from flat base change for the product projections.
This proves the first isomorphism.
Rewriting it as $q_*\circ\Phi_0^{-1}\cong\widehat\Phi_0^{-1}\circ q_*$ and taking left adjoints gives the second isomorphism in \eqref{equation: intertwining with abelian extension}.

Choose an ample line bundle on $A'$ and replace it by a power $L$ such that both $L$ and $q^*L$ are very ample.
Set $G_{A'}=\bigoplus_{j=0}^{\dim A'}L^{\otimes j}$ and $G_P=q^*G_{A'}$.
By \cref{theorem: existence of generator}, $G_{A'}$ is a classical generator, and $G_P$ contains a classical generator as a direct summand because $\dim P\leq\dim A'$.
By \cite[Proposition 2.2(c)]{dimitrov_haiden_katzarkov_kontsevich_2014_dynamical_systems_and_categories}, complexity does not increase under exact functors, so
\begin{align*}
\delta_0(G_P,\Phi_0^nG_P)
=\delta_0(q^*G_{A'},q^*\widehat\Phi_0^{\,n}G_{A'})
\leq\delta_0(G_{A'},\widehat\Phi_0^{\,n}G_{A'}).
\end{align*}
Taking exponential growth rates gives $\hcat(\Phi_0)\leq\hcat(\widehat\Phi_0)$.
\end{proof}

\begin{lem}\label{lemma: numerical spectral radius of the abelian extension}
In \cref{setup: abelian extension of a relative autoequivalence}, the endomorphisms $\numg{\widehat\Phi_0}$ and $\numg{(\Psi\boxtimes\id_{D^b(C)})}$ of $N(A')_{\RB}$ have the same characteristic polynomial.
In particular, $\rho(\numg{\widehat\Phi_0})=\rho(\numg{\Psi})$.
\end{lem}
\begin{proof}
Recall that translations and tensoring by algebraically trivial line bundles act trivially on numerical classes.
By the additivity of the Poincar\'e bundle, replacing $\bar f,\bar g$ by $\bar f-\bar f(0),\bar g-\bar g(0)$ does not change $\numg{\widehat\Phi_0}$.
Thus, we may assume that $\bar f$ and $\bar g$ are homomorphisms.
Put $U=\numg{(\tau_{\bar f})}_*$, $V=\numg{(\tensfunc_{\PC_{\bar g}})}$, and $T=\numg{(\Psi\boxtimes\id_{D^b(C)})}$.

For every integer $n>0$, let $D_n$ be the numerical pullback of $\mu_n=\id_B\times[n]_C$.
Since the deck transformations of $\mu_n$ are translations, finite flat base change gives $\mu_n^*\mu_{n*}=(\deg \mu_n)\cdot\id$ on numerical classes, so $D_n$ is invertible.
The identities $\mu_n\circ \tau_{n\bar f}=\tau_{\bar f}\circ \mu_n$, $\tau_{n\bar f}=\tau_{\bar f}^n$, and $\mu_n^*\PC_{\bar g}\cong\PC_{\bar g}^{\otimes n}$, together with flat base change for the product projections, give $D_nUD_n^{-1}=U^n$, $D_nVD_n^{-1}=V^n$, and $D_nTD_n^{-1}=T$.
Hence
\[
\det(\lambda\cdot\id-UVT)=\det(\lambda\cdot\id-U^nV^nT).
\]

Since $U$ and $V$ are each conjugate to all their positive powers,
they are unipotent.
Put $d=\dim N(A')_{\RB}$ and write $U=\id+N_U$,
where $N_U^d=0$.
Therefore,  the finite expansion $U^n=\sum_{j=0}^{d-1}\binom{n}{j}N_U^j$ and the same discussion for $V$ show that
\[
    F(n,\lambda)\coloneqq\det(\lambda\cdot\id-U^nV^nT)
\]
is polynomial in $n$ and $\lambda$.
The preceding identity gives $F(n,\lambda)=F(1,\lambda)$ for every
positive integer $n$, so $F$ is independent of $n$.
Evaluating at $n=0$, we obtain
\[
    \det(\lambda\cdot\id-UVT)=\det(\lambda\cdot\id-T).
\]
Since $\numg{\widehat\Phi_0}=UVT$, this proves the first assertion.

Finally, \cref{theorem: Gromov Yomdin equality for abelian varieties,lemma: entropy of external product with identity} give
\begin{equation*}
\log\rho(\numg{\widehat\Phi_0})
=\log\rho(T)=\hcat(\Psi\boxtimes\id_{D^b(C)})
=\hcat(\Psi)=\log\rho(\numg{\Psi}).\qedhere
\end{equation*}
\end{proof}

We next identify a numerical subspace on which the action of the abelian factor descends to $X$.
Let $p:P=B\times Y\to X$ be the product cover of \cref{theorem: Kawamata product cover}, and let $i_y:B\hookrightarrow P$ be the inclusion of the fibre over $y\in Y$.

\begin{lem}\label{lemma: descent of numerical fibre classes}
The induced map $i_{y*}:N(B)_{\RB}\to N(P)_{\RB}$ is injective, and its image $W=i_{y*}N(B)_{\RB}$ is independent of $y$.
Moreover, $\Gal(p)$ acts trivially on $W$ and $p_*|_W$ is injective.
\end{lem}
\begin{proof}
For $U,V\in D^b(B)$, adjunction gives $\chi_P(\pi_B^*U,i_{y*}V)=\chi_B(U,V)$.
The nondegeneracy of the numerical Euler pairing proves the injectivity of $i_{y*}$.
The objects $i_{y*}V$ vary in an algebraic family over the connected variety $Y$, so their numerical classes are independent of $y$.

Every $\gamma\in\Gal(p)$ acts by a translation on $B$ and an automorphism on $Y$.
Translations act trivially on $N(B)$, and changing the fibre does not change its numerical class.
Thus, $\Gal(p)$ fixes $W$ pointwise.
For $w\in W$,
\[
p^*p_*w=\sum_{\gamma\in\Gal(p)}\gamma_*w=\abs{\Gal(p)}w.
\]
This proves the injectivity of $p_*|_W$.
\end{proof}

We can now express the entropy of an autoequivalence of $X$ in terms of an autoequivalence of the abelian factor of the product cover.

\begin{thm}\label{theorem: reduction of entropy to abelian factor}
Let $X$ be a smooth projective variety with finite Albanese morphism, and let $p:B\times Y\to X$ be a product cover as in \cref{theorem: Kawamata product cover}.
For every $\Phi\in\Auteq(X)$, there are an integer $k>0$ and an autoequivalence $\Psi\in\Auteq(B)$ such that
\[
k\hcat(\Phi)=\hcat(\Psi)=\log\rho(\numg{\Psi})=k\log\rho(\numg{\Phi}).
\]
\end{thm}
\begin{proof}
Put $P=B\times Y$.
By \cref{proposition: lifting a power to the product cover}, there are an integer $m>0$ and an autoequivalence $\widetilde\Phi$ of $P$ such that $\widetilde\Phi\circ p^*\cong p^*\circ\Phi^m$ and $p_*\circ\widetilde\Phi\cong\Phi^m\circ p_*$.
Choose $r>0$ and the normal form of $\widetilde\Phi^r$ given by \cref{proposition: relative normal form}, with the notation of \cref{setup: abelian extension of a relative autoequivalence}.
By \cref{lemma: categorical entropy of lifts,proposition: basic properties of categorical entropy,proposition: entropy upper bound from abelian factor}, we have
\[
\hcat(\widetilde\Phi^r)=mr\hcat(\Phi)\leq\hcat(\widehat\Phi_0).
\]
By \cref{theorem: Gromov Yomdin equality for abelian varieties,lemma: numerical spectral radius of the abelian extension}, we obtain
\begin{equation}
\label{equation: comparison of catentro of Phi and logrho Psi}
mr\hcat(\Phi)\leq\hcat(\widehat\Phi_0)
=\log\rho(\numg{\widehat\Phi_0})
=\log\rho(\numg{\Psi}).
\end{equation}

Fix $y\in Y$, and let $W=i_{y*}N(B)_{\RB}$ be the subspace in \cref{lemma: descent of numerical fibre classes}.
By \eqref{equation: relative normal form} and the projection formula, for every $E\in D^b(B)$ we have a natural isomorphism
\[
\widetilde\Phi^r(i_{y*}E)
\cong i_{y*}\bigl((t_{f(y)})_*(\Psi(E)\otimes P_{g(y)})\otimes_{\CB}M|_y\bigr)[s],
\]
where $M|_y$ is the one-dimensional fibre of $M$ at $y$.
Here we use $(\Psi\boxtimes\id_{D^b(Y)})\circ i_{y*}\cong i_{y*}\circ\Psi$, $\tau_f\circ i_y=i_y\circ t_{f(y)}$, and $i_y^*\PC_g\cong P_{g(y)}$.
Since translations and tensoring by algebraically trivial line bundles act trivially on numerical classes, it follows that
\[
(\numg{\widetilde\Phi})^r\circ i_{y*}
=(-1)^s i_{y*}\circ\numg{\Psi}.
\]
In particular, $(\numg{\widetilde\Phi})^r$ preserves $W$.
The relation $p_*\circ\widetilde\Phi^r\cong\Phi^{mr}\circ p_*$ gives
\[
(\numg{\Phi})^{mr}\circ p_*\circ i_{y*}
=(-1)^s p_*\circ i_{y*}\circ\numg{\Psi}.
\]
Since $p_*\circ i_{y*}$ is injective by \cref{lemma: descent of numerical fibre classes}, the restriction of $(\numg{\Phi})^{mr}$ to the invariant subspace $p_*W$ is conjugate to $(-1)^s\numg{\Psi}$, and $\rho(\numg{\Psi})\leq\rho((\numg{\Phi})^{mr})$.
Combining this with \cref{theorem: numerical lower bound for categorical entropy} and \eqref{equation: comparison of catentro of Phi and logrho Psi}, we obtain
\[
mr\hcat(\Phi)\leq\log\rho(\numg{\Psi})
\leq mr\log\rho(\numg{\Phi})\leq mr\hcat(\Phi).
\]
Therefore, \cref{theorem: Gromov Yomdin equality for abelian varieties} completes the proof.
\end{proof}

In particular, we have the following result:
\begin{cor}[$=$ \cref{theorem: Gromov Yomdin equality for finite Albanese morphisms}]\label{corollary: Gromov Yomdin equality for finite Albanese morphisms}
Let $X$ be a smooth projective variety with finite Albanese morphism.
For every $\Phi\in\Auteq(X)$ and every $\sigma\in\Stab_N(X)$,
\[
\hcat(\Phi)=h_\sigma(\Phi)=\log\rho(\numg{\Phi}).
\]
\end{cor}
\section{Examples of autoequivalences with positive categorical entropy}\label{section: examples of positive categorical entropy}

\subsection{Preparations for constructing a functor}\label{subsection: Preparations for constructing a functor}
For an abelian variety $A$, let $\PC_A$ be the normalized Poincar\'e line bundle on $A\times\widehat A$, and write
\[
\rderived S_A(-)=\rderived\pi_{2*}(\PC_A\otimes\pi_1^*(-)) \colon D^b(A)\to D^b(\widehat A)
\]
for the Poincar\'e Fourier--Mukai equivalence, where $\pi_1$ and $\pi_2$ are the projections.

First, we review some results and fix notation for the case of principally polarized abelian varieties.

\begin{setup}
\label{setup: pri pol abelian case}
Let $A_0$ be an abelian variety of dimension $g>0$, and let $L_0$ be an ample line bundle defining a principal polarization.
Let $\phi_{L_0}:A_0\xrightarrow{\cong}\widehat A_0$ be the polarization isomorphism $x\mapsto t_x^*L_0\otimes L_0^{-1}$.
Put $\theta=c_1(L_0)$ and $\Phi\coloneqq(-\phi_{L_0})^*\circ\rderived S_{A_0}$.

Put $V_0=\langle1,\theta,\ldots,\theta^g\rangle_{\RB}\subset H^*(A_0,\RB)$.
The cohomological actions of $\tensfunc_{L_0}$ and $\Phi$ preserve $V_0$ and determine a representation
\(
\rho:\SL_2(\ZZ)\to\GL(V_0)
\)
by \cite[(2.2)-(2.4) and Theorem 4.2]{beauville_2010_the_action_of_rm_sl2_on_abelian_varieties}.
Here, $\rho(u)=(\tensfunc_{L_0})^H|_{V_0}$ and $\rho(w)=\Phi^H|_{V_0}$, where
\begin{equation*}
u=\begin{pmatrix}1&1\\0&1\end{pmatrix},\quad
w=\begin{pmatrix}0&-1\\1&0\end{pmatrix}.
\end{equation*}
Applying \cite[Proposition 5.2]{beauville_2010_the_action_of_rm_sl2_on_abelian_varieties} to the unit class and passing to cohomology, we obtain an isomorphism between the representation $\rho$ and the \(g\)-th symmetric power of the standard representation of $\SL_2(\ZZ)$ on $\RB^2$.

Note that in this setting, $V_0$ can be identified with the subspace of $N(A_0)_{\RB}$ spanned by $[L_0^{\otimes j}]$ for $0\leq j\leq g$.
Indeed, the classes $\ch(L_0^{\otimes j})=e^{j\theta}$ for $0\leq j\leq g$ form a basis of $V_0$, and the Euler pairing on $V_0$ is anti-diagonal with nonzero entries in the basis $1,\theta,\ldots,\theta^g$, that is, it is nondegenerate.
Under this identification, the numerical actions of $\tensfunc_{L_0}$ and $\Phi$ agree with their cohomological actions.

Since the cohomological actions depend only on \(\theta\), Beauville’s description obtained from a symmetric representative applies to the given \(L_0\).
\end{setup}

Next, we fix notation for the general case.
\begin{setup}
\label{setup: general abelian variety}
We keep the notation of \cref{setup: pri pol abelian case}.
Let $q:A\to A_0$ be an isogeny of degree $d$, and put $L=q^*L_0$.
Then $L$ is ample and $\chi(L)=d\chi(L_0)=d$ by \cite[Corollary 3.6.6]{book_birkenhake_lange_2004_complex_abelian_varieties}.
Since $L$ is ample, cohomology and base change show that $\rderived S_A(L)$ is a vector bundle of rank $\chi(L)=d$ (\cite[Lemma 14.2.1 and Proposition 14.7.7]{book_birkenhake_lange_2004_complex_abelian_varieties}).

Let $N\coloneqq \det(\rderived S_A(L))^{-1}$ and $\widehat q:\widehat A_0\to\widehat A$ be the dual isogeny.
We claim that 
\begin{equation}\label{equation: determinant of Fourier transform under a principal isogeny}
c_1((-\phi_{L_0})^*\widehat q^*N)=d\theta.
\end{equation}
As $\phi_L^*\rderived S_A(L)\cong H^0(A,L)\otimes L^{-1}$ by \cite[Corollary 14.3.6(a)]{book_birkenhake_lange_2004_complex_abelian_varieties}, taking determinants gives $c_1(\phi_L^*N)=d\,c_1(L)$.
The morphism $\widehat q$ satisfies $\phi_L=\widehat q\circ\phi_{L_0}\circ q$ by \cite[Corollary 2.4.6(d)]{book_birkenhake_lange_2004_complex_abelian_varieties}.
Since $(-\id_A)^*$ acts trivially on $H^2(A,\RB)$, we obtain $q^*c_1((-\phi_{L_0})^*\widehat q^*N)=d\,q^*\theta$.
Moreover, $q^*$ is injective both on cohomology and on $N(A_0)_\RB$.
Therefore, we have the equality \eqref{equation: determinant of Fourier transform under a principal isogeny}.

\end{setup}

Fix a finite subgroup $H\subset A$ and let $L_{A/H}$ be an ample line bundle on $A/H$. 
Its polarization is induced by a principal polarization via an isogeny $A/H\to A_0$ by \cite[Proposition 4.1.2]{book_birkenhake_lange_2004_complex_abelian_varieties}.
Consider the composition $q\colon A\to A/H \to A_0$, which is an isogeny with $H\subset\Ker q$.

Finally, we recall the notion of a \emph{linearization}, which follows the notation in \cite{krug_sosna_2015_equivalences_of_equivariant_derived_categories}.

\begin{rem}
\label{remark: linearizations}
Let a finite group $H$ act on a smooth projective variety $X$ by automorphisms $\lambda_h:X\to X$, with $\lambda_{hk}=\lambda_h\circ\lambda_k$ for $h, k\in H$.
An \emph{$H$-linearization} of $E\in D^b(X)$ is a collection of isomorphisms 
\[
\{\ell_h:E\xrightarrow{\cong}\lambda_h^*E \mid \ell_1=\id_E \text{ and } \ell_{hk}=\lambda_k^*(\ell_h)\circ\ell_k \text{ for } h, k\in H\}.
\]
We use the following two constructions.
\begin{enumerate}
    \item\label{item: linearization hom case} Suppose that $f:X\to T$ is a morphism to a smooth projective variety $T$ such that $f\circ\lambda_h=f$ for every $h\in H$. 
    For $F\in D^b(T)$, the canonical isomorphisms
    \[
        f^*F=(f\circ\lambda_h)^*F\xrightarrow{\sim}\lambda_h^*f^*F
    \] 
    satisfy this condition by functoriality and give $f^*F$ a canonical $H$-linearization.
    \item\label{item: invariant closed subscheme case} If $Z\subset X$ is an $H$-invariant closed subscheme, the isomorphisms $\OO_Z\xrightarrow{\cong}\lambda_h^*\OO_Z$ induced from the structure sheaf $\OO_X$ and the invariance of the ideal sheaf of $Z$ give its canonical linearization. 
\end{enumerate}
Note that tensor products and external products inherit linearizations by tensoring the corresponding isomorphisms.
\end{rem}
\subsection{Construction and proof}\label{subsection: construction and proof}
\begin{prop}\label{proposition: positive categorical entropy on abelian varieties}
Every abelian variety $A$ admits an autoequivalence of positive categorical entropy.

Moreover, for any finite subgroup $H\subset A$, such an autoequivalence can be chosen so that its Fourier--Mukai kernel admits an $H$-linearization for the diagonal action defined by $\lambda_h = t_h\times t_h$ on $A\times A$.
\end{prop}

\begin{proof}
\noindent\textbf{Step 1: Construction on a principally polarized abelian variety.}
Let $A_0$, $L_0$, and $\Phi$ be as in \cref{setup: pri pol abelian case}.
For each integer $k>0$, define
\begin{equation*}
\Psi_0(k)=\tensfunc_{L_0}\circ\Phi^{-1}\circ\tensfunc_{L_0^{-k}}\circ\Phi.
\end{equation*}
The restriction of $\Psi_0(k)^H$ to $V_0$ acts as $\rho(uw^{-1}u^{-k}w)$, 
where
\[
uw^{-1}u^{-k}w=\begin{pmatrix}k+1&1\\k&1\end{pmatrix}.
\]
This matrix has eigenvalues $\lambda_k\coloneqq (k+2+\sqrt{k^2+4k})/2>1$ and $\lambda_k^{-1}$.
Since $\rho$ is isomorphic to the $g$-th symmetric power of the standard representation, the eigenvalues of $\Psi_0(k)^H|_{V_0}$ are $\lambda_k^{g-2j}$ for $0\leq j\leq g$.
In particular, $\lambda_k^g$ is an eigenvalue of $\numg{\Psi_0(k)}$, and \cref{theorem: numerical lower bound for categorical entropy} gives $\hcat(\Psi_0(k))\geq g\log\lambda_k>0$.

\noindent\textbf{Step 2: Construction on an arbitrary abelian variety.}
Choose a principally polarized abelian variety $(A_0,L_0)$ and an isogeny $q:A\to A_0$ with $H\subset\Ker q$, as above.
Retain the notation $L,N,d$, and $\Phi$ of \cref{setup: general abelian variety}.
Define
\[
\Psi\coloneqq \tensfunc_L\circ(\rderived S_A)^{-1}\circ\tensfunc_{N^{-1}}\circ\rderived S_A.
\]
Since it holds
\[
\rderived S_A\circ q^*\cong\widehat q_*\circ\rderived S_{A_0}
\]
by \cite[Proposition 14.7.9(b)]{book_birkenhake_lange_2004_complex_abelian_varieties} and 
the projection formula $\tensfunc_{N^{-1}}\circ\widehat q_*\cong\widehat q_*\circ\tensfunc_{\widehat q^*N^{-1}}$, we obtain the following isomorphisms of functors:
\begin{align*}
\Psi\circ q^*
&\cong q^*\circ\tensfunc_{L_0}\circ(\rderived S_{A_0})^{-1}\circ\tensfunc_{\widehat q^*N^{-1}}\circ\rderived S_{A_0}\\
&\cong q^*\circ\tensfunc_{L_0}\circ\Phi^{-1}\circ\tensfunc_{(-\phi_{L_0})^*\widehat q^*N^{-1}}\circ\Phi.
\end{align*}
For the second isomorphism, we used
\begin{align*}
(\rderived S_{A_0})^{-1}\circ\tensfunc_{\widehat q^*N^{-1}}\circ\rderived S_{A_0}
&\cong \Phi^{-1}\circ(-\phi_{L_0})^*\circ \tensfunc_{\widehat q^*N^{-1}}\circ \rderived S_{A_0}\\
&\cong \Phi^{-1}\circ\tensfunc_{(-\phi_{L_0})^*\widehat q^*N^{-1}}\circ\Phi.
\end{align*}
By \eqref{equation: determinant of Fourier transform under a principal isogeny}, the line bundles $(-\phi_{L_0})^*\widehat q^*N^{-1}$ and $L_0^{-d}$ have the same first Chern class, and hence their tensor functors have the same numerical action.
Therefore, we have the following equality of numerical actions:
\[
\numg{\Psi}\circ q^*=q^*\circ\numg{\Psi_0(d)}.
\]
Since $q^*$ is injective on numerical classes, $\lambda_d^g$ is an eigenvalue of $\numg{\Psi}$.
Thus, \cref{theorem: numerical lower bound for categorical entropy} gives $\hcat(\Psi)\geq g\log\lambda_d>0$.

\noindent\textbf{Step 3: Linearization of the Fourier--Mukai kernel.}
First, we determine the Fourier--Mukai kernel of $\Psi$.
Let $d_A:A\times A\to A$ be a morphism given by $d_A(x,y)=x-y$, and put $Q=\rderived S_{\widehat A}(N^{-1})[g]$.
By \cite[Corollary 14.7.3]{book_birkenhake_lange_2004_complex_abelian_varieties}, the quasi-inverse of $\rderived S_A$ has kernel $\PC_{\widehat A}^{-1}[g]$.
On $A\times\widehat A\times A$, let $r_{ij}$ and $r_2$ denote the projections, and let $p_1,p_2$ be the projections of $\widehat A\times A$.
Define $\nu:A\times\widehat A\times A\to\widehat A\times A$ by $\nu(x,\xi,y)=(\xi,x-y)$.
By \cite[the proof of Corollary 14.1.8]{book_birkenhake_lange_2004_complex_abelian_varieties}, under the canonical identification $\widehat{\widehat A}=A$, we have
\[
r_{12}^*\PC_A\otimes r_{23}^*\PC_{\widehat A}^{-1}\cong\nu^*\PC_{\widehat A}.
\]
Consider the following Cartesian square:
\[\begin{tikzcd}
A\times \widehat A\times A \arrow[r, "\nu"] \arrow[d, "r_{13}"']
    & \widehat A\times A \arrow[d, "p_2"]\\
A\times A \arrow[r, "d_A"']
    & A\rlap{.} \arrow[phantom]{ul}[very near end]{\lrcorner}
\end{tikzcd}\]
Since $r_2=p_1\circ\nu$ and $d_A$ is smooth, flat base change gives the following expression for the kernel of $(\rderived S_A)^{-1}\circ\tensfunc_{N^{-1}}\circ\rderived S_A$:
\begin{align*}
&\rderived r_{13*}\bigl(r_{12}^*\PC_A\otimes r_2^*N^{-1}\otimes r_{23}^*\PC_{\widehat A}^{-1}\bigr)[g]\\
&\cong d_A^*\rderived p_{2*}\bigl(\PC_{\widehat A}\otimes p_1^*N^{-1}\bigr)[g]\\
&=d_A^*Q.
\end{align*}
By the projection formula, composing on the left with $\tensfunc_L$ gives the kernel $K=\pi_2^*L\otimes d_A^*Q$ for $\Psi$, where $\pi_2:A\times A\to A$ is the second projection.

Consider the morphism $f:A\times A\to A_0\times A$ given by $f(x,y)=(q(y),x-y)$.
Since $L=q^*L_0$, we have $K\cong f^*(L_0\boxtimes Q)$.
Since $H\subset\Ker q$, $f\circ(t_h\times t_h)=f$ for every $h\in H$.
Therefore, \eqref{item: linearization hom case} of \cref{remark: linearizations} gives a linearization of $K$ for the diagonal action $t_h\times t_h$. 
This proves the assertion.
\end{proof}

\begin{thm}\label{theorem: positive categorical entropy for finite Albanese morphisms}
Let $X$ be a smooth projective variety whose Albanese morphism $a_X:X\to\Alb(X)$ is finite.
There exists an autoequivalence $\Phi\in\Auteq(X)$ with $\hcat(\Phi)>0$ if and only if $\kappa(X)<\dim X$.
\end{thm}
\begin{proof}
Let $X$ be a variety of general type whose Albanese morphism is finite.
As $X$ contains no rational curve, the canonical sheaf $\omega_X$ is ample (see \cite[7.13 Exercises 8]{book_debarre_2001_higherdimensional_algebraic_geometry} or \cite[Lemma 2.1]{diverio_trapani_2019_quasinegative_holomorphic_sectional_curvature_and_positivity_of_the_canonical_bundle}).
Thus, every autoequivalence of $D^b(X)$ is standard \cite{bondal_orlov_2001_reconstruction_of_a_variety_from_the_derived_category_and_groups_of_autoequivalences} and has zero categorical entropy by the finiteness of $\Aut(X)$ and \cite{kikuta_takahashi_2019_on_the_categorical_entropy_and_the_topological_entropy}.
This shows the necessity.

Suppose that $\kappa(X)<\dim X$.
By \cref{theorem: Kawamata product cover}, there is a finite \'etale Galois cover $p:P=B\times Y\to X$, where $B$ is abelian, $Y$ is of general type, and $\Gal(p)$ acts diagonally by translations on $B$.
In particular, $\dim B=\dim X-\kappa(X)>0$.
Write the action of $\gamma\in\Gal(p)$ as $\gamma(b,y)=(b+a_\gamma,\gamma_Y(y))$, and put $H=\{a_\gamma\mid\gamma\in\Gal(p)\}\subset B$.
By \cref{proposition: positive categorical entropy on abelian varieties}, there is an autoequivalence $\Psi\in\Auteq(B)$ with $\hcat(\Psi)>0$ whose Fourier--Mukai kernel $K$ has a diagonal $H$-linearization.

Consider $\widetilde\Phi=\Psi\boxtimes\id_{D^b(Y)}\in\Auteq(P)$.
Its Fourier--Mukai kernel is $K\boxtimes\OO_{\Delta_Y}$, after reordering the factors.
Since $\gamma_Y\times\gamma_Y$ preserves $\Delta_Y$ as a closed subscheme, $\OO_{\Delta_Y}$ has its canonical $\Gal(p)$-linearization (see \eqref{item: invariant closed subscheme case} of \cref{remark: linearizations}).
Together with the linearization on $K$ induced by $\gamma\mapsto a_\gamma$, this gives $K\boxtimes\OO_{\Delta_Y}$ a diagonal $\Gal(p)$-linearization.
By \cite[Proposition 4.2]{krug_sosna_2015_equivalences_of_equivariant_derived_categories}, $\widetilde\Phi$ descends to an autoequivalence $\Phi\in\Auteq(X)$ satisfying $\widetilde\Phi\circ p^*\cong p^*\circ\Phi$ and $p_*\circ\widetilde\Phi\cong\Phi\circ p_*$.
Hence \cref{lemma: categorical entropy of lifts,lemma: entropy of external product with identity} give
\[
\hcat(\Phi)=\hcat(\widetilde\Phi)=\hcat(\Psi)>0.\qedhere
\]
\end{proof}

\bibliographystyle{amsalpha}
\bibliography{converted_bibtex_tyoshida}

\end{document}